\documentclass[11pt]{article}

\usepackage[a4paper, left=2.7cm,top=2cm,right=2.7cm,bottom=3.5cm]{geometry}
\usepackage{siunitx}
\usepackage[blocks]{authblk}

\usepackage{amsthm}
\usepackage{pifont}
\usepackage{psfrag}
\usepackage{pgfplots}
\usepackage{subcaption}
\usepackage{siunitx}
\usepackage{graphicx}

\usepackage{enumitem}
\setitemize{topsep=0em, label=\scriptsize{$\blacksquare$},wide}
\setenumerate{topsep=0em, label=(\alph*),wide}

\usepackage{cmbright}
\usepackage[T1]{fontenc}

\usepackage{amssymb}
\usepackage{booktabs}

\usepackage{mathtools}

\newtheorem{theorem}{Theorem}[section]
\newtheorem{lemma}[theorem]{Lemma}

\newtheorem{remark}[theorem]{Remark}
\newtheorem{definition}[theorem]{Definition}

\allowdisplaybreaks

 \DeclareMathOperator{\ran}{ran}

\newcommand{\R}{\mathbb R}
\newcommand{\sph}{\mathbb S}

\newcommand{\N}{\mathbb N}

\newcommand{\X}{\mathbb{X}}
\newcommand{\Y}{\mathbb{Y}}
\newcommand{\dom}{\mathbb{D}}

\newcommand{\ZZ}{\mathbb Z}
\newcommand{\XX}{\mathbb X}
\newcommand{\YY}{\mathbb Y}
\newcommand{\QQ}{\mathbb Q}
\newcommand{\WW}{\mathbb W}

\newcommand{\Mo}{\mathbf M}

\newcommand{\Ao}{\mathbf  A}
\newcommand{\Fo}{\mathbf  F}
\newcommand{\Ho}{\mathbf H}
\newcommand{\Uo}{\mathbf U}
\newcommand{\Io}{\mathbf I}
\newcommand{\II}{\boldsymbol{\chi}}
\newcommand{\Po}{\mathbf P}

\newcommand{\rmd}{\mathrm d}
\newcommand{\eps}{\epsilon}

\newcommand{\edot}{\,\cdot\,}

\newcommand\abs[1]{\left\vert#1\right\vert}

\newcommand\norm[1]{\left\Vert#1\right\Vert}
\newcommand\snorm[1]{\Vert#1\Vert}

\newcommand\set[1]{\left\{#1\right\}}

\newcommand\sset[1]{\{#1\}}
\newcommand{\om}{\omega}

\newcommand{\la}{\lambda}

\newcommand{\mua}{\mu_a}
\newcommand{\mus}{\mu_s}
\newcommand{\data}{v}
\newcommand{\nill}{N}

\newcommand{\qi}{q_i}
\newcommand{\qb}{q_o}

\newcommand{\kl}[1]{\left(#1\right)}

\newcommand{\skl}[1]{(#1)}
\newcommand\inner[2]{\left\langle#1,#2\right\rangle}

\newcommand{\req}[1]{(\ref{eq:#1})}

\newcommand{\Om}{\Omega}

\DeclareMathOperator*{\argmin}{arg\,min}

\newcommand{\tfun}{T}
\newcommand{\rfun}{R}

\newcommand{\Lo}{\mathbf  L}
\newcommand{\To}{\mathbf  T}
\newcommand{\Ko}{\mathbf K}
\newcommand{\Jo}{\mathbf J}
\newcommand{\wave}{\mathbf U}

\newcommand{\Fr}{F}
\newcommand{\Gr}{G}

\newcommand{\coloneqq}{:=}

\newcommand{\prox}{\operatorname{prox}}

\allowdisplaybreaks

\numberwithin{equation}{section}
\numberwithin{figure}{section}
\numberwithin{theorem}{section}

\newcommand{\cmaa}[1]{#1}

\usetikzlibrary{shapes,arrows}
\tikzstyle{decision} = [diamond, draw, fill=green!20, 
    text width=5em, text badly centered, node distance=3cm, inner sep=0pt]
\tikzstyle{block} = [rectangle, draw, fill=green!20, 
    text width=7em, text centered, rounded corners, minimum height=4em]
\tikzstyle{sblock} = [rectangle, draw, fill=green!20, 
    text width=7em, text centered, rounded corners, minimum height=2em, node distance=2cm]
\tikzstyle{line} = [draw, -latex']
\tikzstyle{cloud} = [draw, ellipse,fill=red!20, node distance=3.5cm,
    minimum height=2em]
 \usetikzlibrary{decorations, decorations.text,backgrounds}

\title{Stochastic Proximal Gradient Algorithms for  Multi-Source Quantitative Photoacoustic Tomography}

\author{Simon Rabanser}

\affil{Department of Mathematics, University of Innsbruck\\
Technikerstra{\ss}e 13, 6020 Innsbruck, Austria\\
E-mail: {\tt simon.rabanser@uibk.ac.at}}

\author{Lukas Neumann}
\affil{Institute of Basic Sciences in Engineering Science, University of Innsbruck\\
Technikerstra{\ss}e 13, 6020 Innsbruck, Austria\\
E-mail: {\tt lukas.neumann@uibk.ac.at}}

\author{Markus Haltmeier}
\affil{Department of Mathematics, University of Innsbruck\\
Technikerstra{\ss}e 13, 6020 Innsbruck, Austria\\
E-mail: {\tt markus.haltmeier@uibk.ac.at}}

\date{}
\begin{document}

\maketitle

 \begin{abstract}
 The development of accurate and efficient image reconstruction algorithms  is a central aspect of quantitative photoacoustic tomography (QPAT). In this paper, we address this issues for multi-source QPAT using the radiative
transfer equation (RTE) as accurate model for light transport. The tissue  parameters are jointly reconstructed from the acoustical data
measured for each of the applied sources. We develop stochastic proximal gradient methods for multi-source QPAT,
which are more efficient  than  standard proximal gradient methods in which a
single iterative update has complexity  proportional to the number applies sources. Additionally, we introduce a
completely new  formulation of QPAT as multilinear (MULL)  inverse problem which avoids explicitly solving the RTE. The MULL formulation of QPAT is again addressed with stochastic proximal gradient methods. Numerical results for both approaches  are presented. Besides the introduction of stochastic proximal gradient algorithms to QPAT, we consider the new MULL formulation of QPAT as main contribution of this paper.

\medskip\noindent
\textbf{Keywords:} Photoacoustic tomography; image reconstruction; radiative transfer equation; 
multilinear inverse problem; limited view; stochastic gradient method; limited data; Dykstra algorithm

\end{abstract}

\section{Introduction}
\label{sec:intro}

Photoacoustic tomography (PAT)
is an emerging
imaging modality, which combines the benefits
of pure  ultrasound imaging (high resolution)
with those of pure optical tomography (high contrast); see~\cite{Bea11,Wan09b}.
The basic principle of PAT is as follows~(see Figure~\ref{fig:pat}): A semitransparent sample such as a part of a human patient
is illuminated with short pulses of optical radiation. A fraction of the optical energy is absorbed inside the sample, which causes thermal heating, expansion, and a subsequent acoustic pressure wave  depending on the interior absorbing structure of the sample.
The acoustic pressure is measured outside of the  sample and used  to reconstruct an image of the interior.

One important reconstruction problem  in PAT is recovering  the initial pressure distribution  (see, for example,  \cite{AgrKucKun09,BurBauGruHalPal07,Hal14,haltmeier2017iterative,HalSchuSch05,huang2013full,nguyen2016dissipative,RosNtzRaz13}). The initial pressure distribution only provides qualitative information about the tissue-relevant parameters, as {it is} the product of the optical absorption coefficient
and the spatially varying optical intensity, which again indirectly depends on the tissue parameters.  Quantitative photoacoustic tomography (QPAT)  addresses this issue and aims at quantitatively estimating the tissue parameters by supplementing the  inversion of the acoustic wave equation
with an inverse problem for light propagation  (see, for example,
\cite{AmmBosJugKang11,BalJolJug10,BalRen11,CheYa12,CoxArrBea07a,CoxArrKoeBea06,cox2012quantitative,HalNeuRab15,haltmeier2017analysis,KruLiuFanApp95,MamRen14,NatSch14,RenHaoZha13,RosRazNtz10,TarCoxKaiArr12,YaoSunHua10}).

\begin{figure}[htb]
\begin{center}
  \includegraphics[width=1\columnwidth]{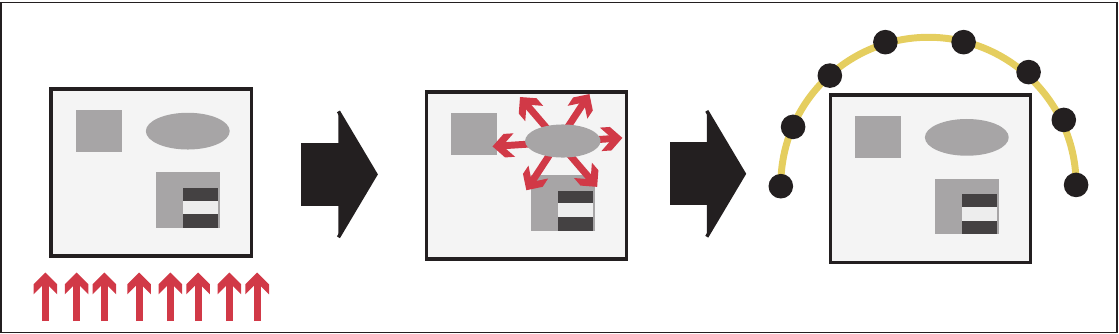}
\caption{\label{fig:pat}{{{Basic principles of PAT.}}
\textbf{Left}: the investigated object is illuminated with a short optical pulse;
\textbf{Middle}: due to the  thermoelastic effect, the absorbed light distribution  induces
an acoustic  pressure wave depending on internal tissue properties;
\textbf{Right}: the acoustic pressure wave is measured outside the  object and used to
reconstruct an image of the interior.}}
\end{center}
\end{figure}

\subsection{Multi-Source QPAT}

In this paper, we consider image reconstruction in QPAT using multiple sources. We allow limited view measurements, where, for each illumination, partial data are collected only from a certain angular domain. For modeling the light transport, we use the radiative transfer equation~(RTE), which~is~commonly considered as a very accurate model for light transport in tissue~(see, for~example,~\cite{Arr99,DaLiVol6,EggSch15,Kan10}). {In particular, opposed to the diffusion approximation, the RTE allows for modeling directed optical radiation, which is required for a reasonable QPAT forward model. Additionally, it allows for including internal voids as regions of low scattering.} As proposed in~\cite{HalNeuRab15}, we work with a single-stage
reconstruction procedure for QPAT, where the optical parameters are reconstructed directly from the measured acoustical data. The image reconstruction problem of multi-source
QPAT using $N$ different sources can be formulated as a system of nonlinear
equations {(see, for example,}~\cite{HalNeuRab15,GaoFenSon15})
\begin{equation} \label{eq:ip}
	\Fo_i (\mu) =   \data_i \quad \text{ for }  i  =1, \dots,  \nill \,.
\end{equation}
Here, $\Fo_i$ is the operator that maps the unknown parameter pair
 $\mu = (\mua, \mus)$ consisting of the absorption coefficient $\mua \colon \Omega  \to \R$ and the scattering coefficient {$\mus \colon \Omega  \to \R$} to the measured  acoustic data $\data_i$  corresponding to the  $i$-th source distribution (see Section~\ref{sec:fwd} for precise definitions).
There are two main {classes of methods for solving} the nonlinear inverse problem \req{ip},
namely, Tikhonov type regularization on the one
and iterative regularization methods on the other hand \cite{Engl96,Kal08,scherzer2009variational}.
Both approaches are based on rewriting \req{ip} as a single equation
$\Fo (\mu) = \data$ with forward operator $\Fo =  (\Fo_i)_{i=1}^\nill$ and
data $\data  = (\data_i)_{i=1}^\nill$.
In Tikhonov regularization, one defines  approximate solutions as minimizers
of  the penalized least squares functional  $\frac{1}{2}  \norm{\Fo  \skl{ \mu}- \data}^2 + \la \rfun (\mu)$.
Here, $\rfun (\, \cdot \,)$ is an appropriate regularization functional included to stabilize the inversion process  and $\la$  a regularization parameter
that has to be carefully  chosen  depending on the data and the noise. In iterative regularization methods, stabilization is achieved
via early stopping of  iterative  schemes.   In such a situation, one usually applies  iterative  optimization  techniques designed   for minimizing the un-regularized
least squares functional {$\frac{1}{2} \norm{\Fo  \skl{ \mu}- \data}^2$}, and the  iteration index
plays the role of the regularization parameter.

Tikhonov type as well as iterative regularization methods can
both be formulated as finding a solution of the optimization problem
\begin{equation} \label{eq:opt}
\left\{\begin{aligned}
&  \min  && \frac{1}{2}
\sum_{i=1}^\nill \norm{\Fo_i  \skl{ \mu}- \data_i}^2  + \Gr(\mu),  \\
&  \text{with }    &&\mu   \in L^2(\Om) \times L^2(\Om)  \,.
\end{aligned}
\right.
\end{equation}
{In iterative} regularization methods, one takes $G   =  \II_{\dom}$, the characteristic function
of the domain of definition $\dom$ of the forward operator (taking the value $0$ in and the value $\infty$ outside of $\dom$).  In Tikhonov regularization, we take $G   =   \II_{\dom}  + \la \rfun$.
Well established algorithms for solving {Equation} \req{opt} are proximal  gradient algorithms \cite{combettes2005signal,bauschke2011convex},
which can be written in the form
  \begin{equation} \label{eq:proxgrad}
	\mu_{k+1}  =  \prox_{s_k G} \kl{ \mu_k
	-  s_k \sum_{i=1}^\nill \Fo_i'(  \mu_k)^*\kl{ \Fo_i( \mu_k ) - \data_i}  } \,.
\end{equation}
Here, $\prox_{s_k G}$ is the proximity operator and
$s_k$ the {positive} step size; $ \Fo_i'( \mu_k )$  denotes the derivative of  the $i$-th forward operator evaluated at $ \mu_k$ with $\Fo_i'( \mu_k)^* $ being its Hilbert space adjoint.

\subsection{Stochastic Proximal Gradient Algorithms}

Each iteration in the proximal gradient algorithm \eqref{eq:proxgrad}
can be numerically quite expensive,  since it requires  solving the
forward and adjoint problems for all $\nill$ equations
in \req{ip}.  In many cases, stochastic (proximal)
gradient  methods turn out to be more efficient since these methods  only consider one of the equations  in
\req{ip} per iteration.
The stochastic proximal gradient method (see, for example,  \cite{bertsekas2010incremental,bertsekas2011incremental,xiao2014proximal,duchi2009efficient,li2017averaged,pereyra2016survey} and the references therein) for solving
\eqref{eq:opt} is  defined by
\begin{equation*} 
	\mu_{k+1}  =  \prox_{s_k G} \kl{ \mu_k
	-  s_k  \Fo_{i(k)}'(  \mu_k)^*\kl{ \Fo_i( \mu_{i(k)} ) - \data_i}  } \,,
\end{equation*}
where $i(k) \in \set{1, \dots, \nill}$ corresponds to one of the equations in \req{ip} that is  selected randomly for the
update in the $k$-th iteration. In opposition to the standard
proximal gradient method, this requires solving only one forward  and  one adjoint problem per iteration. Therefore, one iterative step
is much  cheaper for the stochastic gradient method than for the full  gradient method.
In the case of no regularization, $\la =0$, the stochastic proximal gradient method  reduces to
 the Kaczmarz method for inverse problems studied in~\cite{de2008steepest,haltmeier2007kaczmarz,haltmeier2007kaczmarz2}.

The computationally most expensive task in the above methods is the numerical solution
of the RTE. In this paper, we therefore additionally study a  reformulation of the inverse problem
of QPAT avoiding the computation of a solution of the RTE. For this purpose,
the inverse problem is reformulated as multilinear inverse problem \eqref{eq:mull},
where  the RTE is added as a constraint instead of explicitly including  its solution. The new formulation will be again addressed by Tikhonov regularization in combination with proximal stochastic gradient methods as discussed  in
Section~\ref{sec:ML}.

Note that, in QPAT, it has often  been assumed that  the  initial pressure distribution (corresponding to  each illumination) is already recovered  from acoustic measurements (see, for~example,~\mbox{\cite{alberti2017disjoint,ammari2017multi,BalRen11,CoxArrKoeBea06,SarTarCoxArr13,TarCoxKaiArr12,RenHaoZha13,wang2017iterative,YaoSunHua10})}. Research was focused on inverting the light propagation in tissues either modeled by  the RTE or the diffusion approximation.  In the case that acoustic measurements are {only known on parts of} the boundary, reconstruction  of the  initial pressure distribution is not possible in a stable manner. In order to obtain stable reconstruction results in \cite{HalNeuRab15}, we propose a single-stage approach for QPAT,  where the optical parameters are directly recovered from the acoustic boundary data.
Throughout this paper, we will make use of this approach, which delivers stable results especially in the limited view situation.
In opposition to  \cite{HalNeuRab15}, in this paper, we introduce  (proximal) stochastic gradient methods, which effectively exploit the multi-illumination
structure and  turn out to be faster than the  standard  proximal gradient methods.

\subsection{Outline}

The remainder of this paper is organized as follows.  In Section \ref{sec:fwd}, we provide the mathematical
model for  QPAT  (the forward problem) using the RTE. We allow multiple sources and  partial acoustic
measurements.  We also recall known results for QPAT including differentiability of the forward problem.
In Section  \ref{sec:SPG}, we address the inverse problem  of QPAT using Tikhonov regularization and study the
proximal stochastic gradient method for its solution. The new reformulation of the inverse problem of QPAT as a
multilinear  inverse problem  is presented in Section \ref{sec:ML}. For the solution of the proposed formulation, we again
develop proximal gradient  methods. Numerical results are  presented in Section \ref{sec:num}. The paper is concluded
with a summary and outlook presented in Section~\ref{sec:conclusion}.

\section{The Forward Problem in QPAT}
\label{sec:fwd}

The image reconstruction problem of QPAT can be written as the system \req{ip} of nonlinear equations, where the  forward operators $\Fo_i$ map tissue relevant parameters  to acoustic data sets  recorded in specific regions outside the tissue.
Precise formulations will be given in this section.

\subsection{Mathematical Notation}

We fix some mathematical notation that is used throughout this paper.
We denote by $\Omega \subseteq \R^d$ a convex domain with piecewise smooth boundary modeling
our  domain of interest, with $d \in \set{2,3}$ denoting  the spatial dimension. {In order to be able to impose appropriate boundary conditions for the~RTE, it is  convenient to
split the set} $ \Gamma \coloneqq \partial \Om \times \sph^{d-1}$  into  inflow and  outflow boundaries,
\begin{align*}
\Gamma_-& \coloneqq
\left\{(x,\theta)\in\partial\Om\times \sph^{d-1} \mid \nu(x)\cdot\theta \leq 0\right\} \,,
\\
\Gamma_+&\coloneqq
\left\{(x,\theta)\in\partial\Om\times
\sph^{d-1}
\mid \nu(x)\cdot\theta>0\right\} \,,
\end{align*}
with $\nu(x)$ denoting the outward pointing  unit normal at $x \in \partial \Om$ and $x \edot y$ the standard inner product in $\R^d$.
We write $B_R = \set{x \in \R^d \mid \snorm{x} < R }$ for the ball of radius $R$ centered at the origin and suppose   $B_R \supseteq  \Omega$.

By $L^2\skl{\Omega}$ and $L^2\skl{\Omega\times \sph^{d-1}}$, we denote the Hilbert spaces of square integrable functions  on $\Omega$ and $\Omega\times \sph^{d-1}$,
respectively.
By $ L^2\kl{\Gamma_-, \abs{\nu \cdot \theta}}$, we denote the space of all  $\qb \colon  \Gamma_-\to \R$
for which  $\norm{\qb}^2_{ L^2\kl{\Gamma_-, \abs{\nu \cdot \theta}}} \coloneqq \int_{\Gamma_-}
\abs{\qb (x,\theta)}^2 \abs{\nu \cdot \theta} \rmd (x,\theta)$  is finite.
We further write
\begin{align*}
 \snorm{\Phi}_{\WW}^2
&\coloneqq
\snorm{\Phi}_{L^2\skl{\Om \times \sph^{d-1}}}^2
+\snorm{\theta \cdot \nabla_x \Phi}_{L^2\skl{\Om \times \sph^{d-1}}}^2 +
\snorm{\Phi|_{\Gamma_-}}_{L^2\kl{\Gamma_-,
\abs{\nu \cdot \theta}}}^2 \,, \\
\snorm{v}_{\YY}^2 &\coloneqq
 \int_0^\infty\int_{\partial B_R} \abs{v(x,t)}^2 t \,  \rmd x \rmd t \,,
\end{align*}
 and define
 \begin{align*}
 \QQ  &\coloneqq L^2\skl{\Omega\times \sph^{d-1}} \times L^2\kl{\Gamma_-, \abs{\nu \cdot \theta}}, \\
 \XX &\coloneqq L^2\skl{\Omega} \times L^2\skl{\Omega},  \\
 \WW &\coloneqq \sset{ \Phi \colon \Omega\times \sph^{d-1} \to \R  \mid
 \norm{\Phi}_{\WW} < \infty },
 \\
 \YY &\coloneqq \sset{ v \colon  \partial B_R  \times (0, \infty) \to \R  \mid
 \norm{v}_{\YY} < \infty } \,.
\end{align*}
 {The inner} products in $\QQ$, $\XX$, $\WW$ , $\YY$  will be denoted by $\inner{\edot}{\edot}_\QQ$,
 $\inner{\edot}{\edot}_\XX$, $\inner{\edot}{\edot}_\WW$, $\inner{\edot}{\edot}_\YY$, respectively.
 The subspace of all $\Phi \in \WW$  with $\Phi|_{\Gamma_-}=0$ will be denoted by $\WW_0$.

Elements in $\XX$ will be written in the  form $\mu = (\mua, \mus)$ and are the parameters
 we aim to determine. They are actually required to be contained in the  convex subset
 \begin{equation} \label{eq:dom}
 \dom(\To) \coloneqq  \set{ \mu  \in  \XX  \mid
 0 \leq  \mua \leq  \overline{\mu}_a  \,, 0 \leq \mus \leq  \overline{\mu}_s } \,,
 \end{equation}
where   $\overline{\mu}_a, \overline{\mu}_s >0$. Elements in $\QQ$ will be written in the form $q = (\qb, \qi)$ and model the optical sources. Elements in $\WW$ describe the optical radiation,
and elements in  $\YY$ the  measured   acoustic data.

\subsection{The Radiative Transfer Equation}

To specify the  forward operators, we require mathematical models for the light propagation, the~conversion of optical
into acoustic energy, and the propagation of the acoustic waves.
These models will be presented in the rest of this section.

We model the optical  radiation by a function $\Phi \colon  \Omega \times \sph^{d-1} \to \R$, where  $\Phi(x,\theta)$ is the density
of photons  at position $x \in \Om$ and propagating  into direction $\theta \in \sph^{d-1}$.  The interaction of the photons with the background are described by  absorption coefficient  $\mua \colon \Om \to \R$, the scattering coefficient $\mus \colon \Om \to \R$, and the scattering  operator $\Ko \colon \Phi  \mapsto \Ko \Phi  $, taking the form
({see} \cite{Arr99,Kan10})
\begin{equation}\label{eq:so}
\forall \kl{x, \theta} \in \Om \times \sph^{d-1} \colon \quad
\Ko \Phi(x, \theta) =
\int_{\sph^{d-1}}  k(\theta, \theta' ) \Phi(x, \theta') \rmd \theta' \,,
\end{equation}
with  scattering  kernel  $k \colon \sph^{d-1} \times \sph^{d-1} \to \R$.
The absorption coefficient describes the ability of the background to absorb photons and the scattering coefficient describes the amount of  photon scattering.
The scattering kernel $k\kl{\theta, \theta'}$ describes the redistribution of velocity directions due to interaction of the photons  with the background.
From physical considerations, it is natural to assume $k$ to be measurable, symmetric, nonnegative, and to satisfy
$\int_{\sph^{d-1}}  k \kl{\edot, \theta'} \rmd \theta' = 1$.
In this article, we are concerned with the situation when the kernel is known a priori.
Additionally, we assume $k$ to be essentially bounded.  

The photon density $\Phi(x, \theta)$ is supposed to satisfy the stationary
radiative transfer equation (RTE),
\begin{equation}\label{eq:rte}
 \kl{\theta \cdot \nabla_x
+ \mua  + \mus (\Io-\Ko) }
\Phi(x,\theta)=
\qi (x,\theta)
 \quad \text{ for  }  (x, \theta) \in \Om \times \sph^{d-1}
\end{equation}
with boundary conditions
\begin{equation}\label{eq:rte-b}
\Phi|_{\Gamma_-} (x,\theta)
=
\qb(x,\theta) \quad \text{ for  }  (x, \theta) \in \Gamma_- \,.
\end{equation}
Here, $\qi \colon \Om \times \sph^{d-1} \to \R$ denotes an internal  photon source  and $\qb\colon \Gamma_- \to \R$ a prescribed boundary source pattern.
{Note that PAT uses  very short light pulses (below microseconds) and
that light propagation happens on  time scales much shorter than the scale of acoustic  wave propagation. This~justifies the use of the stationary case for the RTE; see} \cite{BalJolJug10}
{for a more complete discussion.}

\begin{theorem}[Well-posedness of the RTE]  \label{thm:rte}
For every $\mu  \in  \dom(\To) $  and $q \in  \QQ$, the stationary RTE \req{rte} admits a unique solution $\Phi \in \WW$.
Moreover, there exists a constant $C$ only depending on the parameters  $\overline{\mu}_a, \overline{\mu}_s >0$ (defining the domain $\dom(\To) $), such that
\begin{equation} \label{eq:apriori}
 \norm{\Phi}_{\WW}
 \leq
 C\,  \kl{    \norm{\qi}_{L^2} + \norm{\qb}_{L^2} }\,.
\end{equation}
\end{theorem}

\begin{proof}
See \cite{EggSch14b}.
\end{proof}
\begin{definition}[Solution operator for the RTE]\label{def:rte}
The solution operator for the RTE is defined by
\begin{equation*}
\To\colon \QQ \times \dom(\To)
\to \WW   \colon
\kl{q,\mu} \mapsto \To (q,\mu) := \Phi \,,
\end{equation*}
where $\Phi$ denotes the unique solution of \req{rte}.
\end{definition}
Theorem \ref{thm:rte} guarantees that the operator $\To$, mapping $(q,\mu)\in \QQ \times  \dom(\To) $ to the solution of the~RTE, is well defined. Note that in the actual application
$q = (\qi, \qb) \in \QQ $ are prescribed sources, and  $\mu = (\mua, \mus) \in  \dom(\To)$ are the unknown parameters  to  be recovered.

\subsection{Heating Operator}

Due to the spatially varying absorption of photons, the
tissue is locally heated and emits an acoustic pressure wave.
The acoustic source is proportional to the amount of absorbed  photons, the~light intensity and the so-called Gr\"uneisen
parameter $\gamma$  describing the efficiency of  conversion of optical to acoustical energy. {We assume $\gamma$ to be constant and after appropriate re-scaling we take $\gamma =1$; for more details about the Gr\"uneisen parameter, we refer to} \cite{cox2012quantitative}.
Therefore, the conversion of the  optical energy into acoustic pressure wave  is described by the  heating operator defined  as follows.

\begin{definition}[Heating operator]\label{def:heating}
The heating operator is defined by
\begin{align}\label{eq:heating}
\Ho \colon  \QQ \times \dom(\To) &\to L^2(\Om)
\\  \nonumber
(q,\mu )  &\mapsto
\mua \int_{\sph^{d-1}}
\To  (q,\mu)  (\edot,\theta)\rmd  \theta \,.
\end{align}
\end{definition}

If one introduces the averaging operator $\Ao  \colon \WW \mapsto L^2(\Omega)$ defined by  $\Ao \Phi = \int_{\sph^{d-1}}
\Phi(\edot ,\theta)\rmd  \theta$ one may write
the heating operator in the form
$$\Ho(q,\mu) =  \mua \, \Ao \circ \To (q,\mu)
\quad \text{ for } (q,\mu) \in \QQ \times \dom(\To) \,. $$
Because $\To(q,\mu)$ models  the   photon density,  $\Ao \circ \To (q,\mu)$  actually models the total light intensity. \mbox{The heating} operator  is therefore given by the product   of the  absorption coefficient  and the light intensity.
The averaging operator $\Ao$ is well defined and bounded and therefore the heating operator is well defined as a mapping between $\QQ \times \dom(\To)$ and $L^2(\Omega)$.

\subsection{The Wave Equation}

The local heating causes an acoustic pressure wave, where the initial pressure distribution $p_0$ is
proportional to a fraction of the absorbed energy.   Assuming  constant speed of sound and after~rescaling, the   induced acoustic pressure $p \colon \R^d \times \kl{0, \infty} \to \R$
satisfies  the free-space wave equation:
\begin{equation}  \label{eq:wave}
	\left\{ \begin{array}{ll}
	 (\partial_t^2   - \Delta )p(x,t)
	=
	0
	 & \text{for }\kl{x,t} \in \R^d \times \kl{0, \infty},
	\\
	p\kl{x,0}
	=
	\text{{$p_0(x)$}}
	& \text{for }	x \in \R^d,
	\\
	\partial_t
	p\kl{x,0}
	=0
	&  \text{for } x \in \R^d \,.
\end{array} \right.
\end{equation}
Here, the function $p_0$  vanishes outside $B_R$, the ball of radius $R$, and  acoustic data are collected
on a subset of $\partial B_R \times (0, \infty)$ that we denote by $\Lambda \times (0, \infty)$. {Recall that  coupling of the RTE and the wave equation happens in such a way that the result of the  heating operator $\Ho(q,\mu)$ acts as initial  sound source $p_0$ depending on tissue parameters; see Definition~{\ref{def:fwd}}.}
Standard  existence and uniqueness theory for hyperbolic equations  guarantees that, for any $p_0 \in H^1$,  \req{wave} has a unique solution $p \in H^1$, which
continuously depends on $p_0$. Taking the trace results in loss of regularity by degree $1/2$. Therefore, $p_0 \mapsto  p|_{\partial B_R \times (0, \infty)}$ is continuous between $H^1$ and $H^{1/2}$. The following Lemma implies the much  stronger result that $p_0 \mapsto  p|_{\partial B_R \times (0, \infty)}$ is actually an $L^2$-isometry.

\begin{lemma} \label{thm:trace}
Let $p_0 \in C^\infty \kl{\R^d}$ have support in $B_R$ and  let $p$ denote the solution of \req{wave}.
Then,
\begin{equation}\label{eq:trace}
	\int_{B_R}\int_{0}^\infty
	\abs{p(x,t)}^2  t \rmd t  \, \rmd x  =
	\frac{R}{2} \int_{B_R} \abs{p_0(x)}^2 \rmd x \,.
\end{equation}
\end{lemma}

\begin{proof}
See \cite{FinPatRak04} for  $d$ odd and~\cite{FinHalRak07} for $d$ even.
\end{proof}

\begin{definition}\label{def:wave}
We define the  solution operator with full boundary data for the wave Equation \req{wave} by
\begin{equation} \label{eq:solW}
	\wave \colon C^\infty \skl{B_R} \subseteq L^2\skl{B_R}
	\to \YY \colon
	p_0   \mapsto   p|_{\partial B_R \times (0, \infty)}  \,,
\end{equation}
where $p$ denotes the solution of \eqref{eq:wave}.
\end{definition}
According to Lemma \ref{thm:trace}, the  operator
$\wave$ can be uniquely extended to a bounded linear operator defined on $L^2\skl{B}$, denoted again by $ \wave  \colon  L^2\skl{B_R} \to  \YY$. The
partial acoustic measurements made on $\Lambda \subseteq \partial B_R$
are then modeled by $  \chi_{\Lambda \times (0, \infty)} \wave p_0$.

\subsection{Analysis of the  Forward Problem in Multi-Source QPAT}

We assume that we perform $N$ individual experiments, where each experiment consists of separate optical sources and separate acoustic measurements.
For the $i$-th experiment, we denote the source term by
$q_i \in \QQ$ and assume the  acoustic measurements are
made on $\Lambda_i \times (0, T_i)  \subseteq  \partial B_R \times (0, \infty)$.

\begin{definition}\label{def:fwd}
For any $i \in \sset{1, \dots, \nill}$, we denote
 \begin{align*}
 \To_i  \colon & \dom(\To)  \to \WW \colon  \mu \mapsto \To(q_i,\mu),
 \\
\Ho_i  \colon & \dom(\To)  \to L^2(\Om ) \colon  \mu \mapsto \mu_a (\Ao \circ \To_i) (\mu),
\\
\wave_i  \colon & L^2(\Om) \to \YY \colon  p_0
\mapsto \chi_{\Lambda_i \times (0, T_i)} \wave p_0,
\\
\Fo_i \colon & \dom(\To)  \to \YY \colon \mu \mapsto (\wave_i \circ \Ho_i)(\mu) \,.
\end{align*}
Here, $\To_i$ denotes the $i$-th solution operator  for the RTE,
$\Ho_i$ the  $i$-th heating operator, $\wave_i$ the $i$-th partial solution operator for the wave equation,
and  $\Fo_i$ the $i$-th forward operator.
\end{definition}

Recall that $\To$ stands for the solution operator  for the  RTE \req{rte} given in Definition \ref{def:rte}, \mbox{$\Ao \Phi
= \int_{\sph^{d-1}} \Phi(\edot, \theta) \rmd \theta$} is the averaging operator, and $\wave$
the solution operator for the wave Equation~\req{wave};
see Definition \ref{def:wave}.
The operator $\To_i$ models the photon transport and  its solution (via the  heating operator) acts  as input for the solution of the wave equation and thereby
couples the optical with the acoustical~part.

Next, we recall  continuity  and differentiability  of
the forward operators. For that purpose, we call $h \in  \XX$ a  feasible direction at $\mu \in \dom(\To)$, if there exists some $\eps >0$ with  $\mu  +  \eps h \in  \dom(\To)$.

\begin{theorem}[Continuity\label{thm:F-cont-diff} and Differentiability]\mbox{}

\begin{enumerate}
 \item The operators $\To_i$,  $\Fo_i$ and $\Ho_i$ are sequentially continuous and  Lipschitz-continuous.
 \item For every  $\mu \in \dom(\To) $, the one-sided directional  derivatives  $\To_i' (\mu)(h) $,  $\Fo_i' (\mu)(h) $ of $\To_i$, $\Fo_i$ at $\mu $ in any feasible direction $h$ exist, and are  given by
\begin{align} \label{eq:derT}
\To_i'  (\mu)(h)
& =
 \To \kl{0, -(h_a + h_s - h_s\Ko) \To(\mu),\mu },
  \\ \label{eq:derF}
\Fo_i' (\mu)(h)
&=
\wave_{\Lambda_i,T_i}  \kl{  h_a \Ao\To_i  (\mu)   +
  \mua \Ao (\To_i'  (\mu)(h)) }  \,.
\end{align}
\end{enumerate}
\end{theorem}

\begin{proof}
See \cite{HalNeuRab15}.
\end{proof}

Equations \req{derT} and \req{derF}
define a bounded linear
operator $\Fo_i' (\mu) \colon \XX \to \YY$,
which we call the  derivative of $\Fo_i$ at $\mu \in \dom(\To)$.
Numerical minimization schemes  actually require  the
adjoint of  $\Fo_i' (\mu)$, which we compute next.

\begin{theorem}[Adjoint of $\Fo_i' (\mu)$]\label{thm:grad}
Let $i \in \sset{1, \dots, \nill}$ and $\mu \in \dom(\To)$.
Furthermore, set $\Phi_i \coloneqq  \To_i(\mu)$ and let
$\Phi_i^{*}$ denote the solution of the adjoint problem
\begin{align}\label{eq:phidag}
\kl{ -\theta \cdot \nabla_x
+
\kl{ \mua + \mus - \mus\Ko} } \Phi_i^{*}
 = - \Ao^\ast \kl{ \mua  (\wave_i^* v ) }
\end{align}
with $\Phi_i^{*}|_{\Gamma_+} = 0$.
Then, $\Fo_i' (\mu)^* \colon \YY \to \XX $ is given by
\begin{equation}\label{eq:adjoint}
	\Fo_i' (\mu)^* v =
	\left[\begin{array}{l}
	\Ao ( \Phi_i^* \Phi_i) + (\Ao \Phi_i)  (\wave_i^* v) \\
	\Ao \bigl( [  (\Io- \Ko)( \Phi_i^*) ] \Phi_i \bigr)
	\end{array} \right] \,.
\end{equation}
 \end{theorem}

\begin{proof}
See \cite{HalNeuRab15}.
\end{proof}

Given data $\data_1, \dots, \data_{\nill} \in  \Y$, most numerical schemes  for QPAT use gradients of the  partial data-fidelity terms $\Fr_I  \colon \dom(\To) \to  \R$  for $I  \subseteq \sset{1, \dots, \nill}$, where
\begin{equation}\label{diffF}
	\Fr_I(\mu)
  = \sum_{i \in I} \Fr_i(\mu)
\quad\text{ with }  \quad 	\Fr_i(\mu)
	 \coloneqq
	  \frac{1}{2}
	 \norm{ \Fo_i (\mu) -  \data_i}_{\YY}^{2} \,.
\end{equation}
By the chain rule, the gradient  of  $\Fr_I $ is given by
$ \nabla \Fr_I(\mu)  = \sum_{i \in I} \nabla \Fr_i(\mu) $ with
$\nabla \Fr_i(\mu) = \Fo_i'(\mu)^* \kl{\Fo_i(\mu) - \data_i} $, where
$\Fo_i'(\mu)^*  $  can be computed  by Theorem \ref{thm:grad}.
Convergence of schemes such as the (stochastic) proximal gradient method {considered in} the following section  require the Lipschitz continuity of $\nabla \Fr_I$, which will be shown in the following theorem.

\begin{theorem}[Lipschitz continuity  of $\nabla \Fr_I$]
For any data $\data_1, \dots, \data_\nill \in  \Y$  and
any subset $I  \subseteq \sset{1, \dots, \nill}$, \mbox{the  map} $\mu \mapsto \nabla \Fr_I(\mu)$ 
is Lipschitz-continuous.
\end{theorem}

\begin{proof} Without loss of generality, we assume $\nill=1$, $I =  \sset{1}$ and write $\data  =\data_1$, $\Fr = \Fr_{\sset{1}}$,
$\To=\To_1$, 
$\wave=\wave_1$, 
and $v(\mu) = \Fo(\mu)-\data$.
For any $\mu \in \dom(\To)$, let $\To^*(\mu)$  denote the solution of \eqref{eq:phidag}
with $v(\mu)$ in place of $v$.
Then, for any $\mu, \tilde{\mu} \in \dom(\To)$,
\begin{multline} \label{eq:aux-lip}
 \norm{\nabla \Fr(\mu)-\nabla \Fr(\tilde{\mu}) }^2_{\X} \\ =
  \Vert \Ao( \To^*(\mu) \To(\mu)) + (\Ao \To(\mu)) (\wave^* v(\mu))-
  \Ao ( \To^*(\tilde{\mu}) \To(\tilde{\mu})) - (\Ao \To(\tilde{\mu})) (\wave^* v(\tilde{\mu}))\Vert^2_{L^2(\Om)}\\
  + \norm{
  \Ao \bigl( ( \Io- \Ko)( \To^*(\mu)) \To(\mu) \bigr) -   
  \Ao \bigl( ( \Io- \Ko)( \To^*(\tilde{\mu}) )\To(\tilde{\mu}) \bigr)
  }^2_{L^2(\Om)} \,.
\end{multline}
The  second term in \eqref{eq:aux-lip} can be bounded by 
\begin{multline*}
2\norm{   \Ao \bigl( ( \Io- \Ko)( \To^*(\mu)) \To(\mu) \bigr) -   
  \Ao \bigl( ( \Io- \Ko)( \To^*(\mu ) )\To(\tilde{\mu}) \bigr)}_{L^2(\Om)}^2 \\
+
2 \norm{ \Ao \bigl( ( \Io- \Ko)( \To^*(\mu)) \To(\tilde{\mu}) \bigr) 
+    
\Ao \bigl( ( \Io- \Ko)( \To^*(\tilde{\mu}) )\To(\tilde{\mu}) \bigr)}_{L^2(\Om)}^2
\\ \leq  4 \kl{\norm{k}_\infty + 1}^2 \abs{\sph^{d-1}}   \norm{\To^*(\mu)}_{L^\infty(\Om \times \sph^{d-1})}^2
   \norm{ \To(\mu)-\To(\tilde{\mu}) }_{L^2(\Om)}^2  \,.
\end{multline*}
The squared norm of the difference 
$\Ao( \To^*(\mu) \To(\mu)) - \Ao ( \To^*(\tilde{\mu}) \To(\tilde{\mu}))$ in the first term in \eqref{eq:aux-lip} is estimated in a similar manner. Furthermore, we have 
\begin{multline*}
 \norm{ (\Ao \To(\mu)) (\wave^* v(\mu))-(\Ao \To(\tilde{\mu})) (\wave^* v(\mu))
 +(\Ao \To(\tilde{\mu})) (\wave^* v(\mu))-(\Ao \To(\tilde{\mu})) (\wave^* v(\tilde{\mu})) }_{L^2(\Om)}^2 \\
 \leq 2  \norm{ (\wave^* v(\mu))\Ao [\To(\mu)-\To(\tilde{\mu})] }_{L^2(\Om)}^2 
 + 2\norm{\Ao \To(\tilde{\mu})(\wave^* (v(\mu)-v(\tilde{\mu})))}^2_{L^2(\Om)}  \,.
\end{multline*}
Noting that $\Ao$, $\wave$ and $\wave^*$ are linear and bounded, 
Theorem \ref{thm:F-cont-diff}, an analog of Theorem \ref{thm:rte} for the supremums 
norm,  and the  computations  above yield the Lipschitz continuity of $\nabla \Fr$.
\end{proof}

\section{The Stochastic Proximal Gradient Method for QPAT}
\label{sec:SPG}
\subsection{Formulation of the  Inverse Problem}

The inverse problem of multi-source QPAT consists in finding
$\mu^\star \in \XX$ from measured data
\begin{align} \label{eq:ipi}
\data_i     = \Fo_i  \skl{ \mu^\star} + z_i   \quad \text{ for } i = 1, \dots ,  \nill \,.
\end{align}
Here, $\mu^\star = \skl{ \mua^\star, \mus^\star}$ are the unknowns to be estimated, $z_i$ are the unknown error vectors, and $ v_1, \dots,  v_N $ are the given noisy data.
Using the notation
\begin{align*}
v &\coloneqq (v_1, \dots, v_N) \in \YY^\nill,\\
\Fo &\coloneqq (\Fo_1,  \dots, \Fo_N) \colon \dom(\To)\to \YY^\nill,
\end{align*}
we can write \req{ipi} in the alternative form
\begin{align} \label{eq:ip2}
 \text{Estimate $\mu^*  \in \XX$ from   }
  v      = \Fo   \skl{ \mu^*} + z      \,.
\end{align}
Here,  $z \in  \YY^N$ denotes the error vector.

There are, at least, two different strategies to address such an  inverse problem: Tikhonov type regularization on  the one  and
iterative methods on the other hand. In this section, we give an overview of such methods. In particular, we describe
proximal stochastic gradient methods (for minimizing the Tikhonov functional),
which seem particularly well suited for multi-source QPAT but have not been investigated
yet for that purpose.

\subsection{Tikhonov Regularization in QPAT}

In this section, we consider a quadratic Tikhonov regularization term for solving \req{ipi}.
Let
\begin{align*}
\Lo \colon \dom(\Lo) \subseteq \XX \to \ZZ&\colon \mu \mapsto \Lo \mu   \
\end{align*}
be a linear, densely defined, and possibly unbounded operator between $\XX$ and another Hilbert space $\kl{\ZZ, \inner{\edot}{\edot}_\ZZ}$ and  set $\dom \coloneqq \dom \cap \dom (\Lo)$. In this context, any element $\mu^+ \in \dom$ with   $
\norm{\Lo \mu^+} = \min  \sset{ \norm{\Lo \mu} \mid \Fo(\mu)  = \data }$ is called an
$\norm{\Lo (\edot) }$-minimizing solution of $\Fo \mu = \data$.
Tikhonov regularization  with regularization term $\frac{1}{2}\norm{\Lo \mu }_\ZZ^2 $ consists  in
computing  a minimizer of the generalized Tikhonov functional
$\tfun_{v,\lambda} \colon \XX
\to \R \cup \set{\infty}$, defined  by
\begin{equation}\label{eq:tik}
\tfun_{v,\lambda}(\mu )
\coloneqq
\begin{cases}
\frac{1}{2}
	  \norm{\Fo( \mu) - v}^2
	+ \frac{ \lambda}{2}\norm{\Lo \mu }^2,
	& \text{if } \mu \in \dom,  \\
	\infty, & \text{otherwise} \,.
	\end{cases}
\end{equation}
Here, $\lambda > 0$ denotes the regularization parameter that acts as a trade-off between the data fitting term and stability.

\begin{theorem}[Well-posedness and convergence]\label{thm:tik}\mbox{}
\begin{enumerate}\item
For any $v  \in \YY$  and any
$\la >0$,
the Tikhonov functional  $\tfun_{\la, v}$
has at least one minimizer.
\item
Let $v  \in \ran(\Fo)$,  $(\delta_m)_{m \in \N} \in (0, \infty)^\N$, $(v^m)_{m \in \N} \in \YY^\N$ with $\norm{v  - v^m} \leq \delta_m$.
Suppose further that $(\la_m)_{m \in \N} \in (0, \infty)^\N$
satisfies    $\la_m \to 0$ and $\delta_m^2/\lambda_m \to  0$ as $m \to \infty$.
Then:

\begin{itemize}[leftmargin=13 mm,labelsep=3mm]
\item Every sequence $(\mu^m)_{m\in\N}$ with $\mu^m \in  \argmin \tfun_{v^m,\lambda_m}$ has a weakly converging subsequence.
\item
The limit  of every weakly  convergent subsequence of $(\mu^m)_{m \in \N}$ is an $\norm{\Lo (\edot) }$-minimizing solution of $\Fo \mu = \data$.

\item If the  $\norm{\Lo (\edot) }$-minimizing solution  of $\Fo \mu = \data$ is unique and denoted by $\mu^+$, then $(\mu^m) \rightharpoonup \mu^+$.
\end{itemize}
\end{enumerate}
\end{theorem}

\begin{proof}
See \cite{HalNeuRab15}.
\end{proof}

\subsection{The  Proximal Stochastic Gradient Algorithm for QPAT}

Depending on the particular  choice of $\Lo$, the Tikhonov functional \req{tik}  may be   ill-conditioned. To address this  issue  in \cite{HalNeuRab15},
 we proposed the   proximal gradient algorithm for minimizing \req{tik}, which~is a very flexible algorithm for minimizing functionals of the form  $\Fr + \Gr$,
where $\Fr$ is smooth and $\Gr$ is convex
(see, for example,  {{\cite{combettes2005signal,bauschke2011convex}}}).  Here, we extend the approach to
the proximal stochastic gradient algorithm. Additionally, we  propose
computing the proximal step using Dykstra's projection algorithm.

\begin{itemize}
\item \textbf{Proximal gradient algorithm:}  The proximal gradient algorithm is a splitting method that iteratively computes explicit gradient steps for $\Fr$ and implicit proximal steps for $\Gr$. In our context, we take $\Fr$ as the data fidelity term and
\begin{equation}\label{eq:G-la}
    \Gr (\mu) =
    \Gr_\la(\mu) \coloneqq
    g_\la(\mu) + \II_{\dom}(\mu)
    \coloneqq
     \frac{\la}{2} \norm{\Lo \mu}^2 + \II_{\dom}(\mu) \,,
    \end{equation}
 where   $\II_{\dom}$ is the characteristic function taking the value zero inside $\dom$ and $\infty$ outside.
 The proximal gradient algorithm for minimizing the QPAT-Tikhonov functional
\eqref{eq:tik} reads
\begin{equation} \label{eq:PGla}
\mu^{k+1} = \prox_{s_k \Gr_\la}
\Bigl( \mu^k -  s_k \sum_{i=1}^\nill  \nabla  \Fr_i(\mu^k)    \Bigr)   \,.
\end{equation}
Here, $\prox_{s_k\Gr_\la}  \colon \XX \to \dom$ denotes the proximal mapping  corresponding to the functional $s_k\Gr_\la$,
\begin{equation}\label{eq:prox}
 \prox_{s_k\Gr_\la}(x) = \argmin
 \set{\frac{1}{2}\norm{ x  - (\edot)}^2  +
 s_k\Gr_\la } \,.
\end{equation}
Furthermore, $\nabla  \Fr_i(\mu^k)  $ is the gradient of the $i$-th data fidelity  term
computed in Theorem~\ref{thm:grad}.

\item \textbf{Dykstra's projection algorithm:}
The constraint quadratic optimization problem \req{prox} can efficiently be solved by a proximal variant of
Dykstra's projection
algorithm \cite{combettes2011proximal,combettes2005signal,bauschke2011convex}.
 For that purpose, we write
$s_k  \Gr_\la  =  \II_{\dom} + g$ with $g(x) \coloneqq \frac{s_k \lambda}{2}\norm{\Lo x}^2$.
Setting $x_0 = \mu$, $p_0=0$ and $q_0=0$,  Dykstra's projection algorithm for \req{prox} reads, for $m \in \N$,
\begin{align}
 y_m &= \prox_g (x_m+p_m),   \label{eq:dykstra1} \\
 x_{m+1} &=  \Po_{\dom} (y_m+q_m),  \label{eq:dykstra2} \\
 p_{m+1} &= x_m+p_m-y_m, \label{eq:dykstra3} \\
 q_{m+1} &= y_m+q_m-x_{m+1}.  \label{eq:dykstra4} 
\end{align}
Both proximal mapping in \req{dykstra1}  and the
projection in \req{dykstra2} can
be computed explicitly. In fact, one readily verifies that
\begin{align}
\prox_g(x)   &  = \kl{ \Io_{\X} + s_k \la \, \Lo^*\Lo }^{-1}x,  \\
\Po_{\dom}( \mu) &= \min\set{\overline \mu, \max\set{0, \mu }} \,.
\end{align}
Here, $\Io_{\X} $ is the identity operator on $\X$ and $\Po_{\dom}$ the projection
onto $\dom$.

\item  \textbf{Proximal stochastic gradient algorithm:}
The methods described  so far require in any iterative step the computation of the  full gradient
\begin{equation*}
\nabla \Fr(\mu)
= \sum_{i=1}^\nill \nabla \Fr_i(\mu)
\quad \text{ with }   \quad
\nabla \Fr_i(\mu) =  \Fo_i'(\mu)^* \kl{\Fo_i(\mu) - \data_i} \,.
\end{equation*}
{The evaluation} of each $\nabla \Fr_i(\mu) $ requires  the solution of the RTE and an adjoint problem and therefore is quite time-consuming. For multi-source QPAT, where $\nill>1$, a significant
acceleration may be obtained by a Kaczmarz strategy,
where in each iterative step only  one of  the summands
$\nabla \Fr_i(\mu) $ is used.
The resulting proximal stochastic gradient method for minimizing the
Tikhonov functional  \req{tik} in QPAT~reads
\begin{equation} \label{eq:proxsg}
    \mu^{k+1}  =
    \text{{$\prox_{s_k G_\la }$}} \kl{ \mu^k
    -  s_k     \nabla  \Fr_{i(k)}(\mu^k)    } \,,
\end{equation}
where $i(k) \in \set{1, \dots, \nill}$ is selected randomly for
the update in the $k$-th iteration.  Furthermore, $\prox_{s_k\Gr_\la}$  is the
proximal mapping  of $s_k\Gr_\la$  that can be computed by Dykstra algorithm
\req{dykstra1}--\req{dykstra4}  and $\nabla  \Fr_i(\mu^k)  $ is the gradient of the $i$-th data fidelity  term 
that can be computed by Theorem \ref{thm:grad}.

{One can also incorporate a block-iterative  (or mini-batch) strategy in the stochastic gradient method, meaning that a small subset of
$\{1,\dots,N\}$ of equations is used per iteration instead of a \mbox{single one}.  Such a variant could be  especially useful in the case of a  large number of different illumination patterns.
For more details about stochastic gradient methods,  see}    \cite{bertsekas2010incremental,bertsekas2011incremental,xiao2014proximal,duchi2009efficient,li2017averaged,pereyra2016survey} {and the references therein.}
{Note that, in general, convergence of  stochastic gradient methods requires asymptotically vanishing  step size} {\cite{pereyra2016survey}}.
\end{itemize}

\subsection{Iterative Regularization Methods}

An alternative class of algorithms  to address nonlinear inverse problems are iterative techniques.
The most basic  iterative method for solving the nonlinear inverse problem
$v =\Fo (\mu ) $ is the Landweber iteration. In the case that the  domain of definition $\dom$ is a proper subset, we have to combine the Landweber iteration with a projection step onto
$\dom$ as presented in this subsection.
The projected Landweber iteration applied  to multi-source QPAT reads
\begin{equation} \label{eq:p-landweber}
\mu^{k+1} =  \Po_{\dom}
\Bigl( \mu^k -  s_k \sum_{i=1}^\nill  \nabla  \Fr_i(\mu^k)  \Bigr)   \,,
\end{equation}
where $\nabla  \Fr_i $ is the gradient of $\Fr_i$ (see Equation~\eqref{diffF}),  and $\Po_{\dom}(\mu) = \min\set{\overline \mu, \max\set{0, \mu }}$
denotes the projection onto $\dom$.
In Tikhonov regularization, the regularity of solutions is enforced by  an explicitly included penalty.
In opposition to that, in iterative regularization methods, a stabilization effect is enforced by early stopping of the iteration.
A common stopping rule  is the discrepancy principle, where  iteration is stopped at the
smallest  index $k\in \N$  satisfying
$\snorm{v - \Fo(\mu^k)} \leq \tau\delta,$
where $\delta$  is an estimate for the noise and $\tau  \geq 1$.
Formally, the projected Landweber iteration  \eqref{eq:p-landweber}
arises as a special case of the proximal gradient iteration \eqref{eq:PGla} for
minimizing the Tikhonov functional, where the regularization parameter is taken as $\la=0$
and where the  proximal mapping  \eqref{eq:prox} reduces to the
orthogonal projection onto $\dom$.

In a similar manner, one can also use a stochastic version of the
projected Landweber iteration.
Using the loping strategy of \cite{de2008steepest,haltmeier2007kaczmarz,haltmeier2007kaczmarz2}  in order to stabilize the iterative process, the resulting  projected loping Landweber--Kaczmarz iteration  reads
\begin{align} \label{eq:plk}
\mu^{k+1} &= \Po_{\dom}
\Bigl( \mu^k -  s_k \om_k  \nabla  \Fr_{i(k)} (\mu^k)  \Bigr)   \,,
 \\ \label{eq:skip}
\om_k   &\coloneqq
\begin{cases}
      1,  & \snorm{ \Fo_{i(k)}(\mu^k) - \data_{i(k)}}_{\X}
      > \tau \delta_{i(k)}, \\
      0,  & \text{otherwise.}
\end{cases} \,
\end{align}
Here, $i(k) \in \set{1, \dots,  \nill}$ for any  $k\in\N$  may be randomly selected, $\tau > 1$
is an appropriately chosen  positive constant,
$\nabla  \Fr_i(\edot)$ is the gradient of the $i$-th data fidelity
term computed in Theorem \ref{thm:grad} and
$\Po_{\dom}(\mu) = \min\set{\overline \mu, \max\set{0, \mu }}$
denotes the projection onto $\dom$.
The {iteration} \req{plk}, \req{skip} terminates  if  $\snorm{\Fo_{i}(\mu^k) - \data_i} \leq \tau \delta$
for all $i \in \set{1, \dots, \nill}$.
It is worth mentioning that, for noise free data, we have $\om_k = 1$ for all $k$
and, therefore, in this special situation, the iteration  becomes
$   \mu^{k+1} = \Po_{\dom}  ( \mu^k -  s_k  \nabla  \Fr_{i(k)} (\mu^k)  ) $, which formally arises from the proximal stochastic  gradient  method \eqref{eq:proxsg} with  $\la=0$.
A convergence analysis of the loping Landweber--Kaczmarz method can be found in \cite{de2008steepest,haltmeier2007kaczmarz}.

\section{QPAT as Multilinear Inverse Problem}
\label{sec:ML}

Since the RTE is time-consuming to solve, we are looking for a suitable reformulation of the inverse problem in multi-source QPAT avoiding computation of a solution of the RTE in each iterative step. In this paper, we propose  to write  \req{ipi} as a multilinear inverse problem, where we add the RTE as a constraint instead of explicitly including its solution. The new formulation will again be
addressed by Tikhonov regularization and proximal stochastic gradient methods.

\subsection{Reformulation  as Multilinear Inverse Problem}

Recall the forward problem of QPAT governed by the RTE \eqref{eq:rte}.
With the abbreviation \mbox{$\Mo(\mu):=  \theta \cdot \nabla_x+ \mua  + \mus (\Io-\Ko)$},
the RTE can be written in compact form $\Mo(\mu) \Phi = q$, where $\mu = (\mu_a, \mu_s)$ is the unknown parameter pair.
In the case of exact data, the  multi-source problem in QPAT  \eqref{eq:ipi} then can
be reformulated as the problem of finding  the tuple
$\mathbf{z} \coloneqq  (\mu, (\Phi_i,H_i)_{i=1}^\nill) \in  \dom  \times (\WW
\times  L^2(\Om) )^\nill$ such that
\begin{equation}\label{eq:mull}
\begin{aligned}
\Mo(\mu)\Phi_i
& =  q_i && \text{ for } i = 1, \dots,\nill,
\\
 H_i
& =  \mu_a \, \Ao \Phi_i
&& \text{ for } i = 1, \dots,\nill,
\\
v_i & =  \wave(H_i)
&& \text{ for } i = 1, \dots, \nill  \,.
\end{aligned}
\end{equation}
Here, the index $i$ indicates  the $i$-th illumination, and
\mbox{$q_i \in \QQ$, $\Phi_i \in  \WW$,  $H_i \in L^2(\Om)$,
$v_i \in \Y$}  are the  corresponding  source, photon density, heating and
acoustical data, respectively, and $\Ao \Phi_i = \int_{\sph^{d-1}}
\Phi_i(\edot ,\theta)\rmd  \theta$ is the  averaging operator.
We call  \eqref{eq:mull} and resulting formulations below the multilinear
(MULL) formulation of QPAT.

\subsection{Application of Tikhonov Regularization}

In the case that the data $v_i$ are only known approximately, we    use
Tikhonov regularization for the stable solution of \eqref{eq:mull}. For that
purpose, we approximate \eqref{eq:mull} by the constrained optimization problem
\begin{equation}\label{eq:consTik}
\begin{aligned}
& \text{{$\min_{(\mu,\Phi_i,H_i)_{i=1}^{N}} $}}
 \kl{ \frac{1}{2}\sum_{i=1}^\nill \norm{v_i - \wave(H_i)}^2 + \frac{\la}{2}\norm{\Lo(\mu)}^2 +
\II_{\dom}(\mu) }, \\
&\text{s.t. }
\begin{cases}
\Mo(\mu)\Phi_i = q_i, \\
H_i  = \mu_a \, \Ao \Phi_i  \; \text{ for } i = 1, \dots, \nill \,.
\end{cases}
\end{aligned}
\end{equation}
Here, the operator $\Lo \mu = (\Lo_a\mu_a,\Lo_s \mu_s)$ is possibly unbounded, $\frac{\la}{2} \snorm{\Lo(\mu)}^2$ is the regularization term   and
$\la >0$ the regularization parameter.
Note that {\eqref{eq:consTik}} is equivalent to  {\eqref{eq:tik}} and therefore the well-posedness and convergence results of Theorem {\ref{thm:tik}} apply to
{\eqref{eq:consTik}} as well.

The constrained optimization problem {\eqref{eq:consTik}} proposed in this paper can be addressed by various solution
methods, for example using penalty methods or augmented Lagrangian techniques \cite{ito2008lagrange}.
\mbox{{In this paper,}} we use a  penalty approach for solving \eqref{eq:consTik}  where the constraints
are included as penalty~term. To simplify notation, we introduce the unconstraint  functionals
\begin{multline}\label{eq:functional}
\Jo^{(i)}(\mathbf{z}):= \frac{a_1}{2}\norm{\Mo(\mu)\Phi_i-q_i}^2+\frac{a_2}{2}\norm{\mu_a \, \Ao \Phi_i - H_i}^2\\
+\frac{a_3}{2}\norm{v_i-\wave(H_i)}^2
+\frac{\la}{2}\norm{\Lo(\mu)}^2 + \II_{\dom}(\mu) \,,
\end{multline}
for certain parameters $a_1, a_2, a_3 >0$ and $\mathbf{z}_i \coloneqq  (\mu, \Phi_i,H_i) \in  \QQ  \times   \WW \times  L^2(\Om) $. The sum of the unconstraint functionals \eqref{eq:functional} over all illuminations will actually be minimized in our numerical implementations.
For that purpose, we define
\begin{align*}
\Jo_1^{(i)}(\mathbf{z}) &= \frac{1}{2} \norm{\Mo(\mu)\Phi_i-q_i}^2, \\
\Jo_2^{(i)}(\mathbf{z}) &= \frac{1}{2} \norm{\mu_a \, \Ao \Phi_i   - H_i}^2,\\
\Jo_3^{(i)}(\mathbf{z}) &= \frac{1}{2} \norm{v_i-\wave(H_i)}^2,\\
\Jo_4^{(i)}(\mathbf{z}) &= \frac{1}{2} \norm{\Lo(\mu)}^2 \,.
\end{align*}
Then, we have $\Jo^{(i)}(\mathbf{z}) = \sum_{\ell=1}^4 a_\ell \Jo^{(i)}_\ell(\mathbf{z})+ \II_{\dom}(\mu)$.
For the approximate  solution of  \req{consTik}, we minimize  the unconstrained functional
 $\Jo (\mathbf{z}) = \sum_{i=1}^4    \Jo^{(i)} (\mathbf{z}) $, which can be written in the forms
\begin{align} \label{eq:pen}
\Jo(\mathbf{z}) &= \sum_{i=1}^\nill  \sum_{\ell=1}^4 a_\ell \Jo^{(i)}_\ell(\mathbf{z}) +
\II_{\dom}(\mu),
\\ \label{eq:pen2}
\Jo(\mathbf{z})  &= \sum_{i=1}^\nill  \sum_{\ell=1}^3 a_\ell\Jo^{(i)}_\ell(\mathbf{z})
+  \frac{\la}{2}\norm{\Lo(\mu)}^2   +  \II_{\dom}(\mu).
\end{align}
(Here and below, we also write $a_4 =  \la$, if it simplifies notation.)
The formulations \eqref{eq:pen} as well as \eqref{eq:pen2} can be
solved by various optimization techniques.  In particular, as  the functionals
are given as the sum of simpler terms, the  stochastic (proximal) gradient method
is particularly appealing.

\subsection{Solution of the MULL Formulation of QPAT Using
Stochastic Gradient Methods}

For solving QPAT in the novel MULL formulation \eqref{eq:mull}, we use stochastic gradient
 methods similar to previous sections. For that purpose, we require the gradients (determining the steepest descent directions) of the
individual functionals $\Jo^{(i)}_\ell(\mathbf{z}_i)$ with  respect to $\mathbf{z}_i=(\mu,\Phi_i,H_i)$, which are given as
  \begin{align} \label{eq:grad-m1}
 \nabla_{\mu_a}\Jo_1^{(i)}(\mathbf{z}) &= \Phi_i (\Mo(\mu)\Phi_i-q_i), \\ \label{eq:grad-m2}
 \nabla_{\mu_s}\Jo_1^{(i)}(\mathbf{z}) &= (\Io-\Ko)\Phi_i (\Mo(\mu)\Phi_i-q_i), \\ \label{eq:grad-m3}
 \nabla_{\Phi_i}\Jo_1^{(i)}(\mathbf{z}) &= \Mo(\mu) (\Mo(\mu)\Phi_i-q_i), \\ \label{eq:grad-m4}
 \nabla_{\mu_a}\Jo_2^{(i)}(\mathbf{z}) &=  (\Ao \Phi_i) (\mu_a\Ao \Phi_i-H_i), \\ \label{eq:grad-m5}
 \nabla_{H_i}\Jo_2^{(i)}(\mathbf{z}) &= -(\mu_a\Ao \Phi_i-H_i), \\ \label{eq:grad-m6}
 \nabla_{\Phi_i}\Jo_2^{(i)}(\mathbf{z}) &= \Ao^{*} [\mu_a(\mu_a\Ao \Phi_i-H_i)], \\ \label{eq:grad-m7}
 \nabla_{H_i}\Jo_3^{(i)}(\mathbf{z}) &= -\wave^T (v_i-\wave(H_i)), \\ \label{eq:grad-m8}
 \nabla_{\mu_a}\Jo_4^{(i)}(\mathbf{z}) &= \Lo_a^{*}\Lo_a\mu_a, \\ \label{eq:grad-m9}
 \nabla_{\mu_s}\Jo_4^{(i)}(\mathbf{z}) &= \Lo_s^{*}\Lo_s\mu_s.  \,
 \end{align}
(All other partial gradients are  vanishing.)  In the following, let $\nill$ be the number of illuminations, write  $\mathbf{z} = (\mu_a,\mu_s, (\Phi_i,H_i)_{i=1}^\nill)$ and  let $(s_k)_{k\in \N}$ be
a sequence of  step sizes. In this paper, we propose the following instances of the stochastic
proximal gradient method for QPAT based on the multilinear formulation
\eqref{eq:mull}.

 \begin{itemize}
 \item \textbf{MULL-projected stochastic gradient algorithm:}
Here, we consider the form \eqref{eq:pen}.
For  any iteration index $k\in \N$ choose $i(k) \in \{1,\ldots,\nill\}$ and $\ell(k) \in \{1,\ldots,4\}$ and define the sequence of iterates  $(\mathbf{z}^{k})_{k \in \N}$ by
 \begin{equation}\label{eq:mullproj}
\mathbf{z}^{k+1} =
	( \Po_{\dom} \times  \Io) \kl{ \mathbf{z}^k - s_k \nabla \Jo_{\ell (k) }^{(i(k))}(\mathbf{z}^k)}.
\end{equation}
Here, the mapping $\Po_{\dom} \times  \Io$ is the proximal mapping corresponding  to
$\mathbf{z} \mapsto \II_{\dom}(\mu),$
which equals the projection $\Po_{\dom}$  in the $\mu$  component
and equals the identity $\Io$ in the other components.

 \item \textbf{MULL-proximal stochastic gradient algorithm:}
Here, we consider the form \eqref{eq:pen2}.
  For  any iteration index $k\in \N$ choose
  $i(k) \in \{1,\ldots,\nill\}$ and $\ell(k)  \in \{1,\ldots,3\}$
  and define sequence of iterates  $(\mathbf{z}^{k})_{k \in \N}$ by
    \begin{equation} \label{eq:mullprox}
\mathbf{z}^{k+1} =
	 \prox_{s_k \Gr_{\la}} \kl{ \mathbf{z}^k - s_k \nabla \Jo_{\ell (k) }^{(i(k))}(\mathbf{z}^k)} \,.
\end{equation}
  {The second step} implements  the proximal mapping of $\mathbf{z} \mapsto s_k \Gr_{\la}(\mu)$ with $\Gr_{\la}(\mu) = \frac{\la}{2}\norm{\Lo(\mu)}^2  +\II_{\dom}(\mu)$. As in the previous section,
this can be computed with Dykstra's projection algorithm \eqref{eq:dykstra1}--\eqref{eq:dykstra4}.
 \end{itemize}

For better scaling, in our actual numerical implementation,  we  replace  the scalar step
sizes $s_k$ by the adaptive step size rule
\begin{equation} \label{eq:adaptive}
s_k^{i,\ell}   \coloneqq  \argmin \{ \mathbf{z}^{k} - t \nabla\Jo_\ell^{(i)}(\mathbf{z}^{k}) \mid t \in \R \} \,.
\end{equation}
{Note that }computing such step sizes does barely increase the computational time of the stochastic gradient method, since all involved calculations are anyhow necessary for computing the gradient for the iterative update. {In opposition to that,
calculating a similar adaptive step size for the algorithms proposed in Section
{\ref{sec:SPG}} would require evaluation of the forward operators $\Fo_i$ and therefore
would significantly increase the computation time. This might be seen as an additional advantage of the novel MULL formulation {\eqref{eq:mull}} and its regularized version {\eqref{eq:consTik}}.}



\section{Numerical Simulations}
\label{sec:num}

For the Tikhonov approach to multi-source QPAT, the radiative transfer equation is numerically solved by a streamline diffusion finite element method. Solving the RTE is required to evaluate the forward operator $\Fo$ and the gradient $\nabla \Fr$ of the data fidelity term in every iterative step. For~the alternative multilinear approach, these calculations are not necessary. However, the application of the transport operator to $\Phi$ has to be calculated for every update of $\Jo_1$. The simulations are performed on the square domain {$\Om = [\SI{-1}{cm},\SI{1}{cm}]^2$}, where the absorption and the scattering coefficient are~supported.

\subsection{Numerical Solution of the RTE}
\label{sec:RTEnum}
Employing a finite element scheme, we derive the weak formulation of Equation \req{rte} by integrating against a test function $ w \colon  \Om \times \sph^1 \to \R$ and replacing the exact solution $\Phi$ by a linear combination in the finite element space $\Phi^{(h)}= \sum_{i=1}^{N_h}c_i^{(h)}\psi_i^{(h)}(x,\theta) $ as in \cite{HalNeuRab15}. Here, the basis function $\psi_i^{(h)}(x,\theta)$ is the product of a basis function in space and a basis function in velocity. The spatial domain is triangulated uniformly with mesh size $h$ and $P_1$-Lagrangian element function for the spatial and velocity domain. By~choosing the test function $w(x,\theta) = \sum_{j=1}^{N_h}w_j(\psi_j(x,\theta) + D(x,\theta)\theta \cdot \nabla_x \psi_j(x,\theta))$ with streamline diffusion coefficient $D(x,\theta)$, we obtain
\begin{multline}\label{weak}
\int_{\Om}\int_{\sph^1} (D\theta \cdot \psi_i-\psi_i)\theta\cdot\nabla_x \psi_j \rmd \theta \rmd x +
  \int_{\Gamma_+} |\theta \cdot \nu| \rmd \sigma \\
  +  \int_{\Om}\int_{\sph^1} (\mu_a+\mu_s-\mu_s \Ko)(\psi_j+D\theta\cdot\nabla_x\psi_j)\psi_i \rmd \theta \rmd x =
  \int_{\Gamma_-} |\theta \cdot \nu|\psi_i \psi_j\rmd \sigma \,.
\end{multline}
Equation (\ref{weak}) yields a system of linear equations $M^{(h)}c^{(h)}=b^{(h)}$, where evaluating the left-hand side of (\ref{weak}) provides the entries of $M^{(h)}$, the right-hand side gives the components of vector $b^{(h)}$. Note~that the sparsity of matrix $M^{(h)}$ is low and solving the linear system for the Tikhonov approach is very time-consuming. On the other hand, the solution via the MULL formulation requires  only a matrix vector multiplication, since in this case $\Phi^{(h)}$ is an independent variable. Thus, only the application to $\Phi$ has to be calculated and the transport equation does not need to be solved.

\subsection{Test Scenario for Multiple Illumination}

The sample is illuminated in orthogonal direction at the boundaries of {$\Om = [-\SI{1}{cm},\SI{1}{cm}]^2$}. \mbox{In our} simulations, we use $\nill=4$ homogenous illuminations and no internal sources. The illuminations are applied separately from each side (left, right, top and bottom) with acoustic data measured  on a half circle on the same side as the illumination (see Figure~\ref{fig:num}). For the scattering kernel, we use the two-dimensional Henyey--Greenstein kernel,
\begin{equation*}
	k(\theta,\theta')
	\coloneqq
	\frac{1}{2\pi}\frac{1-g^2}{1+g^2-2g\cos( \theta \cdot \theta' ) }
	\quad \text{ for } \theta, \theta' \in \sph^1 \,,
\end{equation*}
where the anisotropy factor is chosen as $g=0.5$ in all our experiments.

\begin{figure}[htb]
\centering 
\includegraphics[width=\linewidth]{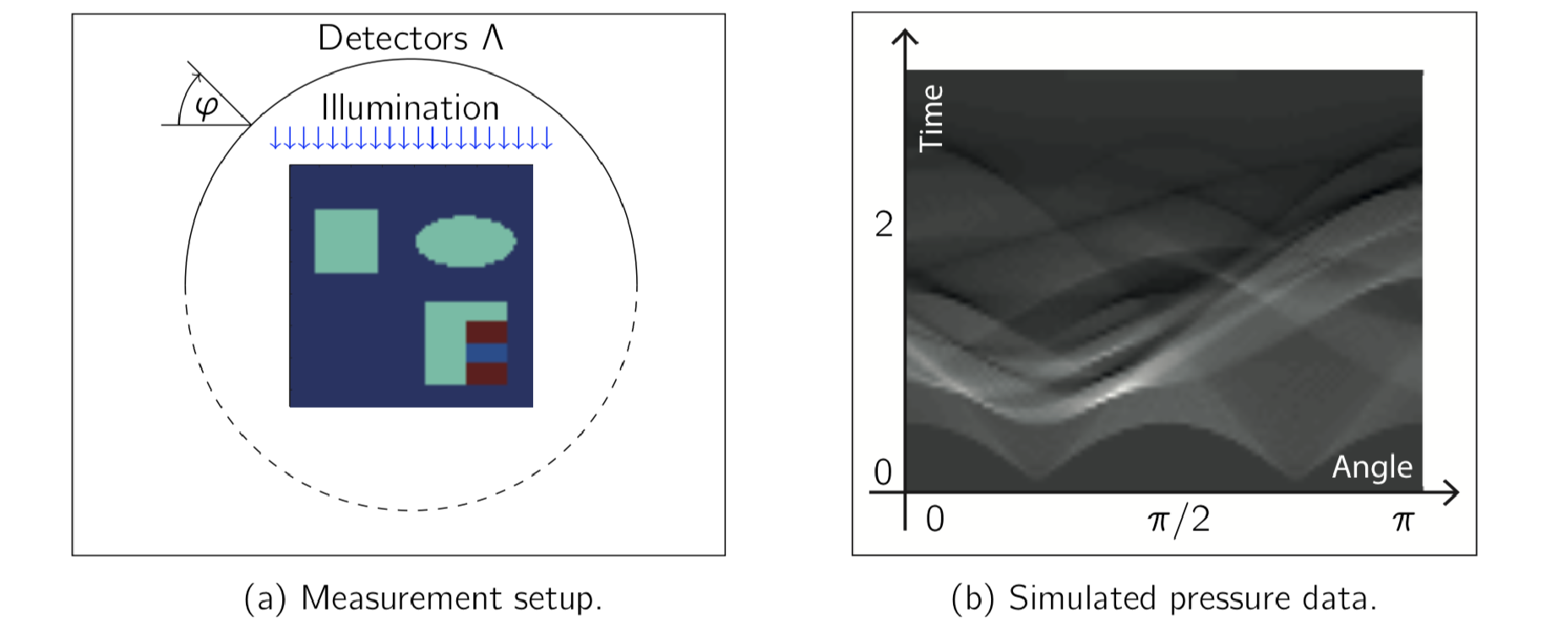}
\caption{(a) The\label{fig:num}  phantom is defined on the square  \cmaa{${\Om = [\SI{-1}{cm},\SI{1}{cm}]^2}$} and the acoustic pressure is measured on a semi-circle
on the side of the illumination. (b) The simulated pressure correspond  to the phantom and the illumination on the right hand side and are represented as gray scale density.}
\end{figure}

For the simulated data, we choose a spatial mesh size $2/100$, in order to discretize the velocity direction the unit circle is divided in $64$ subintervals. In order to avoid inverse crime, for the reconstruction, we use a different spatial mesh  size $h=2/80$ and use $N_\theta = 48 $ velocity directions.
Calculating the simulated data corresponds to evaluating {the forward operators $\Fo_i$
with perpendicular boundary illumination constant along one side of the boundary square,
$\qb(x,\theta) =  \delta(\theta - \theta_i) \chi_{i} (x)\SI{1}{\milli\joule\per\cm}$,
where $\delta$ is an approximation of the Dirac delta function and $\chi_{i}$ the indicator function of side $i$ of $\Om$. In this way, we simulate data}
\begin{equation*}
v_i = \Uo \circ \Ho_i(\mu) + z^{\rm noise}_i  \quad \text{ for } i=1, \dots, 4 \,.
\end{equation*}
{Thereby, the heating operator is computed numerically by solving the RTE as described in  Section~{\ref{sec:RTEnum}}. The wave operator $\wave$ is evaluated by straightforward discretization of the well-known explicit formulas for {\eqref{eq:wave}} that can be found, for example, in} \cite{Joh82,Eva98}.
In the following, we present results for exact data (where $z^{\rm noise}_i=0$) as well as for noisy data.
For the noisy data case, we add $0.5\%$ random noise to the simulated data, i.e., we take the maximum value of the simulated pressure and add white noise
$z^{\rm noise}_i$ with a standard deviation of $0.5\%$ of that maximal value. The phantom, the setup  and the simulated data for one of the four illuminations (top) are  shown in Figures \ref{fig:num} and \ref{fig:phantom}.

\begin{figure}[htb]\centering
\includegraphics[width=0.8\textwidth]{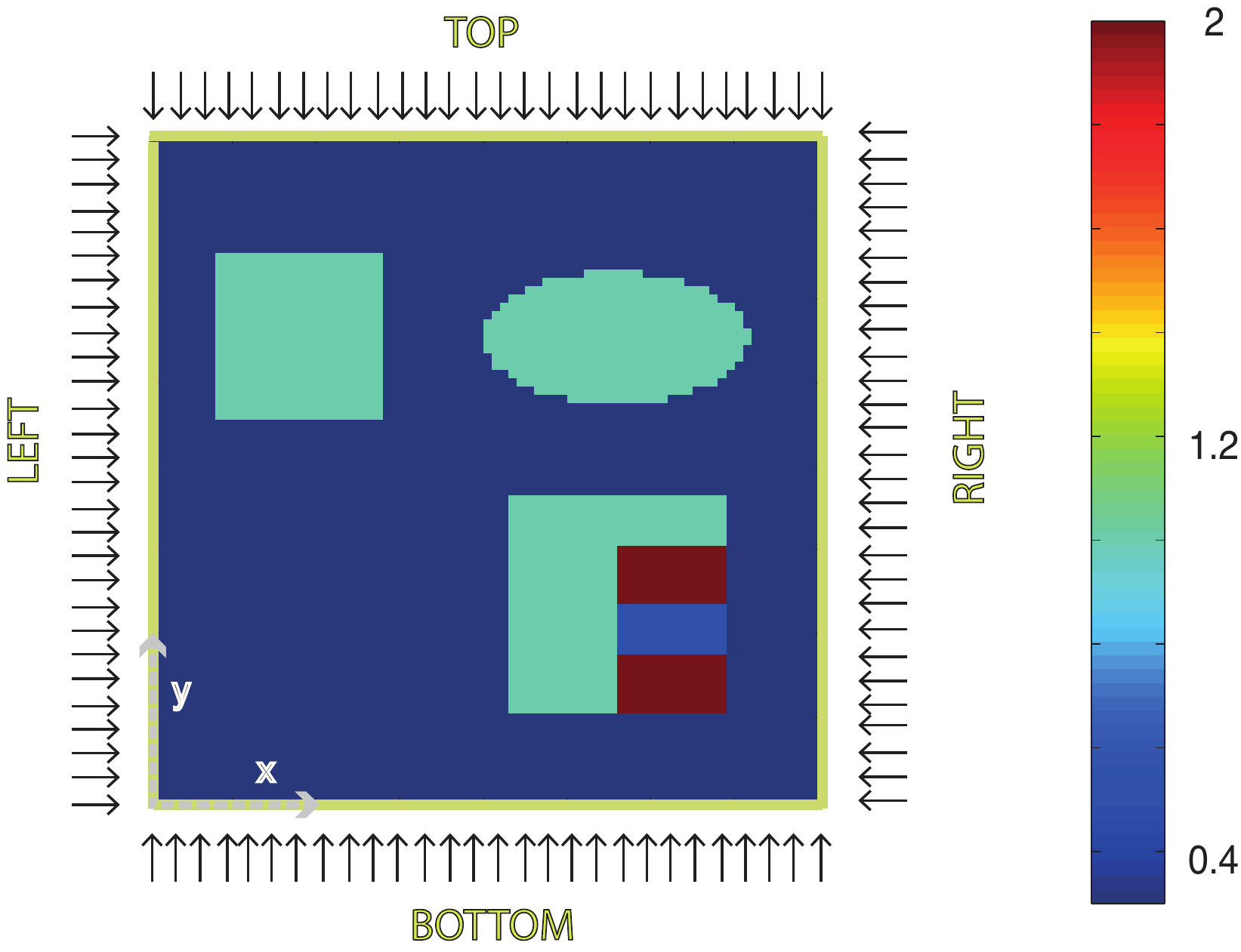}
\caption{\label{fig:phantom} {{Absorption coefficient distribution of the tissue sample used for the numerical examples.}} Background absorption of the tissue is taken as  {${\mu_a=\SI{0.3}{\per \cm}}$}, the blue obstacles have {${\mu_a =\SI{1}{\per\cm}}$} and the red stripes {${\mu_a=\SI{2}{\per\cm}}$}. The area between the red stripes has absorption coefficient {${\mu_a=\SI{0.5}{\per\cm}}$}. The scattering coefficient is constant in the whole sample and chosen to be {${\mu_s =\SI{3}{\per\cm}}$. Illuminations are applied consecutively
from top, right, bottom and left. The corresponding} {boundary sources are given by $\qb(x,\theta) =  \delta(\theta - \theta_i) \chi_{i} (x)\SI{1}{\milli\joule\per\cm}$.} {The $x$- and $y$-axis cover $[\SI{-1}{\cm},\SI{1}{\cm}]$.}}
\end{figure}


\subsection{Numerical Results}

For regularizing the absorption and scattering coefficient, we make use of Laplace regularization and choose $\Lo_a=\Delta $ and $\Lo_s=100\Delta$, respectively.  We assume that the coefficient $\mu$ is known at the boundary of $\Om$ and is therefore used as the starting value of our iterative schemes. Furthermore, we use the boundary value of $\mu$ for regularization; that is, we implement it in the Dykstra projection procedure~\req{dykstra2} by iteratively projecting on the known boundary value. In the following, we discuss the methods that we have outlined in the previous section.

\begin{itemize}
\item \textbf{Standard formulation of QPAT \eqref{eq:ipi}:}
We assume that the scattering coefficient is known and we restrict ourself to reconstructing the absorption coefficient.
Then, the proximal gradient and proximal stochastic gradient algorithm, respectively, read
\begin{align}
\mu_a^{k+1} &= \prox_{s_k \Gr_\la}
\Bigl( \mu_a^k -  \frac{s_k}{4}  \sum_{i=1}^4 \nabla  \Fr_i(\mu_a^k,\mu_s)    \Bigr),
  \\  \label{eq:sg1}
\mu_a^{k+1}  &= \prox_{s_k \Gr_\la}
\Bigl( \mu_a^k -  s_k  \nabla  \Fr_{i(k)}(\mu_a^k,\mu_s)    \Bigr)   \,.
\end{align}
{In contrast,} to the full proximal gradient algorithm, the proximal stochastic gradient algorithm
avoids evaluating the full gradient $\nabla  \Fr$, but selects randomly an illumination number $i\in \{1,\dots,4\}$ for each iterative step.    
Because of formula \req{adjoint}, each iteration of the above procedures requires the calculation of the solution of
the radiative transfer equation $\Phi$ as well of its adjoint $\Phi^{*}$.

The top row in Figure~\ref{fig:IP} shows reconstruction result for the absorption coefficient using the original formulation with  the proximal gradient method with $\lambda = 2 \times 10^{-8}$ and $10$ iterative steps.
The left picture  shows the relative error $ \snorm{\mu_a-\mu_a^k}/{\norm{\mu_a}} $. Note that, in this case, solutions of the RTE and its adjoint have to be computed for four illuminations per iterative step. {The reconstruction results in the bottom row in Figure~{\ref{fig:IP}} are obtained by the proximal stochastic gradient method with $\lambda = 2 \times 10^{-7}$.}
{The regularization parameters $\lambda$ have been selected empirically
as a trade-off between stability and accuracy.}
The total number of iterations is taken as $30$. In each iteration, a illumination pattern is chosen randomly and the computation of RTE and its adjoint is executed only for this single illumination. Therefore, the computational effort for the
proximal stochastic gradient method  is approximately 3/4 of the  proximal gradient algorithm using full gradients.
{For the algorithms  based on the   standard formulation {\eqref{eq:ipi}}, calculating
adaptive step sizes similar to {\eqref{eq:adaptive}} is time-consuming as this requires
another evaluation of  the forward operator $\Fo_i$ and therefore another solution
 of the RTE.
Therefore, we simply use a constant step size rule; in our numerical experiments, it
turned out that $s_k = 0.5$ is a suitable choice.}

\begin{figure}[htb]\centering
\includegraphics[width=\textwidth]{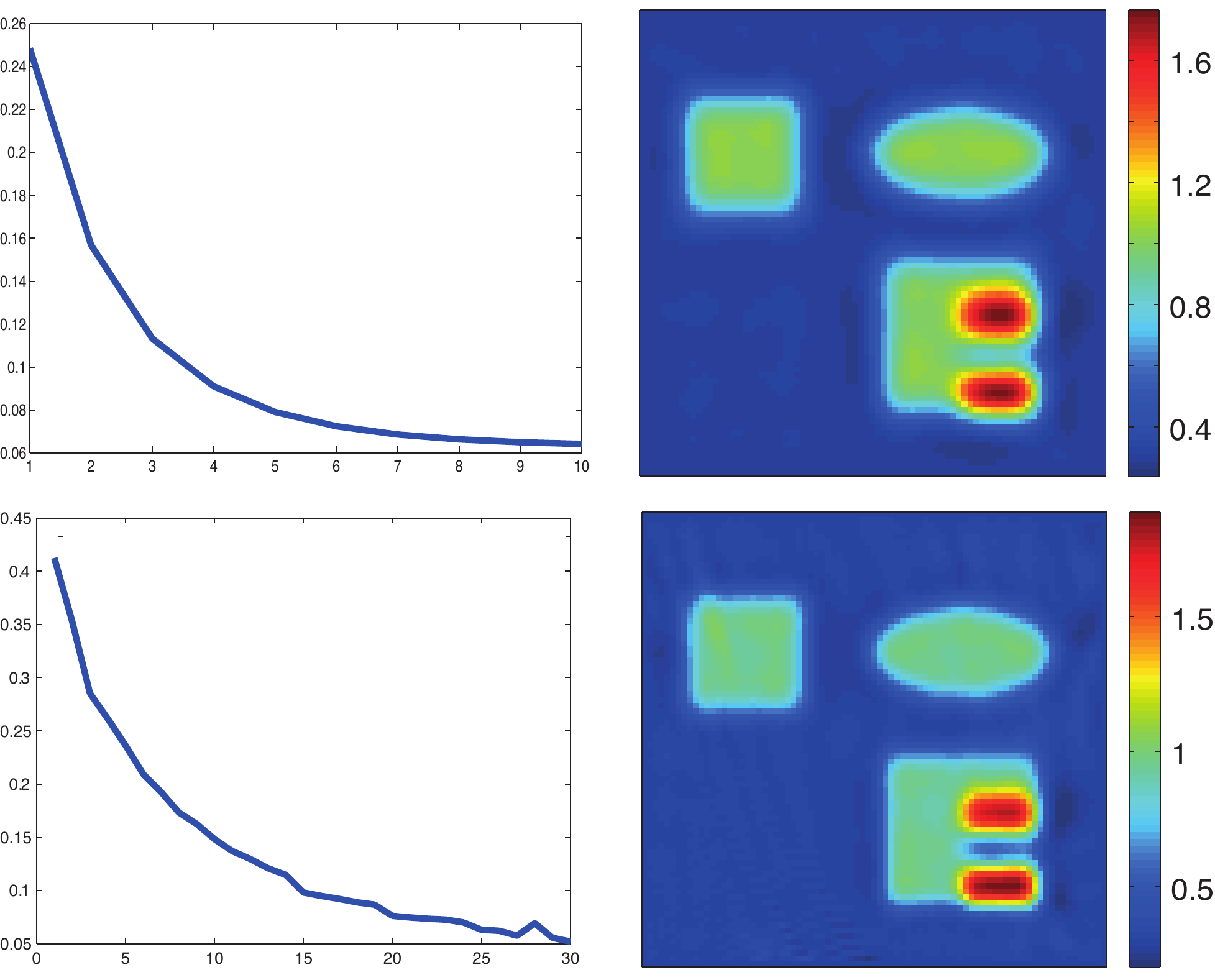}
\caption{{Reconstruction results based on standard  formulation \eqref{eq:ipi}.}\label{fig:IP}
\textbf{Top}: proximal gradient method;
\textbf{Bottom}: proximal stochastic gradient method.
The left images show the relative reconstruction errors of the reconstructed absorption coefficient as
a function of the number of iterations, whereas the right pictures show the result after the final iteration. {(The phantom is as described in Figure~\ref{fig:phantom}.)}}
\end{figure}

\begin{figure}[htb]
\centering
\begin{tikzpicture}[node distance = 2cm, auto]
      \node [sblock] (init) {initialize $\mu_a^0$};
        \node [cloud, left of=init] (data) {input $(v_i)_{i=1}^N$};
    \node [sblock, below of=init, node distance=1.5cm] (identify) {$k \coloneqq 0$};
    \node [sblock, below of=identify, node distance=1.8cm] (select) {select\\
    $i \in \set{1,\dots,\nill}$};
    \node [sblock, below of=select, node distance=1.8cm]  (gradient) {evaluate $\nabla \Fr_{i(k)}(\mu_a^k) $};
    \node [sblock, below of=gradient, node distance=1.8cm]  (update) {update $\mu_a^{k+1}$ by \eqref{eq:sg1}};
     \node [sblock, left of=update, node distance=3cm] (kupdate) {$k \coloneqq k+1$};
    \node [decision, below of=update, node distance=2.3cm] (decide) {$k=k_{\rm max}$?};
    \node [cloud, below of=decide, node distance=2.3cm] (output) {output $\mu_a^{k+1}$};
    \path [line] (init) -- (identify);
    \path [line] (identify) -- (select);
    \path [line] (select) -- (gradient);
    \path [line] (gradient) -- (update);
    \path [line] (update) -- (decide);
    \path [line] (decide) -| node [near start] {no} (kupdate);
    \path [line] (kupdate) |- (select);
    \path [line] (decide) -- node {yes}(output);
    \path [line,dashed] (data) -- (init);
\end{tikzpicture}\hfill
\begin{tikzpicture}[node distance = 2cm, auto]
      \node [block] (init) {initialize $(\mu,\Phi_i, H_i)^0$};
        \node [cloud, right of=init] (data) {input $(v_i)_{i=1}^N$};
    \node [sblock, below of=init, node distance=1.5cm] (identify) {$k \coloneqq 0$};
    \node [sblock, below of=identify, node distance=1.5cm] (select) {select\\
    $i \in \set{1, \dots, \nill}$\\
    $\ell \in \set{1,2,3}$};
    \node [block, below of=select, node distance=2cm]  (update) {update  $(\mu,\Phi_i,H_i)^{k+1}$ by  \eqref{eq:itJ}};
   \node [sblock, below of=update, node distance=1.8cm]  (dykstra) {If $\ell \in\{1,2\}$:  appl. Dykstra};
     \node [sblock, right  of=dykstra, node distance=3cm] (kupdate) {$k \coloneqq k+1$};
    \node [decision, below of=dykstra, node distance=2.1cm] (decide) {$k=k_{\rm max}$?};
    \node [cloud, below of=decide, node distance=2.3cm] (output) {output $(\mu,\Phi_i,H_i)^{k+1}$};
    \path [line] (init) -- (identify);
    \path [line] (identify) -- (select);
    \path [line] (select) --  (update);
    \path [line] (update) -- (dykstra);
    \path [line] (dykstra) -- (decide);
    \path [line] (decide) -| node [near start] {no} (kupdate);
    \path [line] (kupdate) |- (select);
    \path [line] (decide) -- node {yes}(output);
    \path [line,dashed] (data) -- (init);
\end{tikzpicture}
\caption{Flowcharts\label{fig:flowcharts} of  stochastic gradient algorithms for QPAT proposed in this paper.  \textbf{Left}:~algorithm based on the standard formulation \eqref{eq:ipi}.
\textbf{Right}: algorithm based on the novel MULL formulation  \eqref{eq:mull}.
The update  \eqref{eq:sg1} using the standard formulation requires solving the forward RTE and the adjoint RTE, which is not required by  \eqref{eq:itJ} with the  MULL formulation. Simulations are performed with $\nill=4$. }
\end{figure}
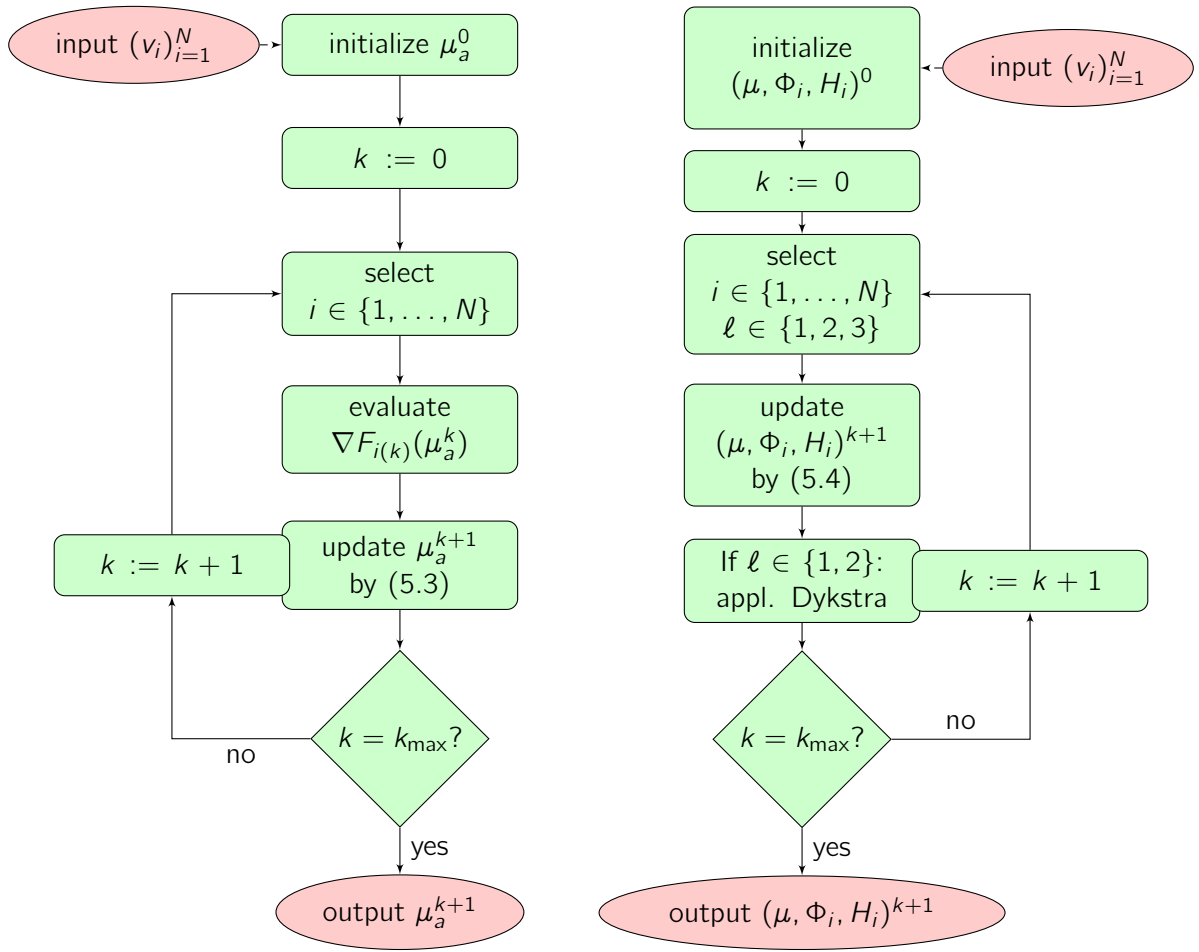

\item \textbf{Novel MULL formulation of QPAT  \eqref{eq:mull}:}
The multilinear approach overcomes the problem of solving the RTE by minimizing  \eqref{eq:pen} or \eqref{eq:pen2}. In both cases, one selects an arbitrary functional and performs  a steepest descent step, resulting in an iterative scheme for the variables $\mu_a$, $\mu_s$, $\Phi$ and $H$. Recall that none of the partial gradients
\eqref{eq:grad-m1}--\eqref{eq:grad-m9} requires solving the  RTE (which is the most
time-consuming part for the standard formulation of QPAT).
In each iterative step, we take a random illumination number $i \in \{1, \dots, 4\}$ and a random functional number $\ell  \in \{1,2,3\}$. The gradient step then consists of the update rule
\begin{equation}\label{eq:itJ}
 \begin{pmatrix}
 \mu\\
 \Phi_i\\
 H_i
  \end{pmatrix}^{k+1}=
   \begin{pmatrix}
 \mu\\
 \Phi_i\\
 H_i
  \end{pmatrix}^{k}
  +s_k \cdot \nabla \Jo_\ell((\mu,\Phi_i,H_i)^k) \,.
\end{equation}
{Dykstra's algorithm}  for smoothing the $\mu$ component is applied after each iterative
step when $\ell\in\{1,2\}$.
Iteration \req{itJ}  contains a gradient step for the RTE. Since one gradient step is not enough to obtain an appropriate approximation to the solution of the transport equation, we apply iteration \req{itJ} $40$ times whenever $\ell=1$ is chosen. In this situation, we apply the Dykstra iteration in the $\mu$ component after these 40 iteration steps, whereas the positivity projection is done in every step. Flowcharts of  the stochastic gradient algorithms (standard and MULL formulations) are shown in Figure~\ref{fig:flowcharts}. 
For the projected stochastic gradient method, regularization of $\mu$ is done by incorporating the regularization functional $\Jo_4$ in the random choice of functionals; see \eqref{eq:pen}. The positivity restriction is realized with the cut projection
$\Po_{\dom} (\mu) = \max\{0,\mu\}$ applied after every iterative step.

Figure~\ref{fig:MLIP} shows reconstructions with the stochastic gradient  methods for the novel MULL formulation of   QPAT \eqref{eq:mull}. The results in the top row are for the MULL-proximal stochastic gradient  algorithm with $\lambda = 2 \times 10^{-8}$  and in the the bottom row results  for the MULL-projected stochastic gradient method with  $\lambda = 2 \times 10^{-8}$ are shown. In both cases, we used $1000$ iterations.

\begin{figure}[htb]\centering
\includegraphics[width=0.85\textwidth]{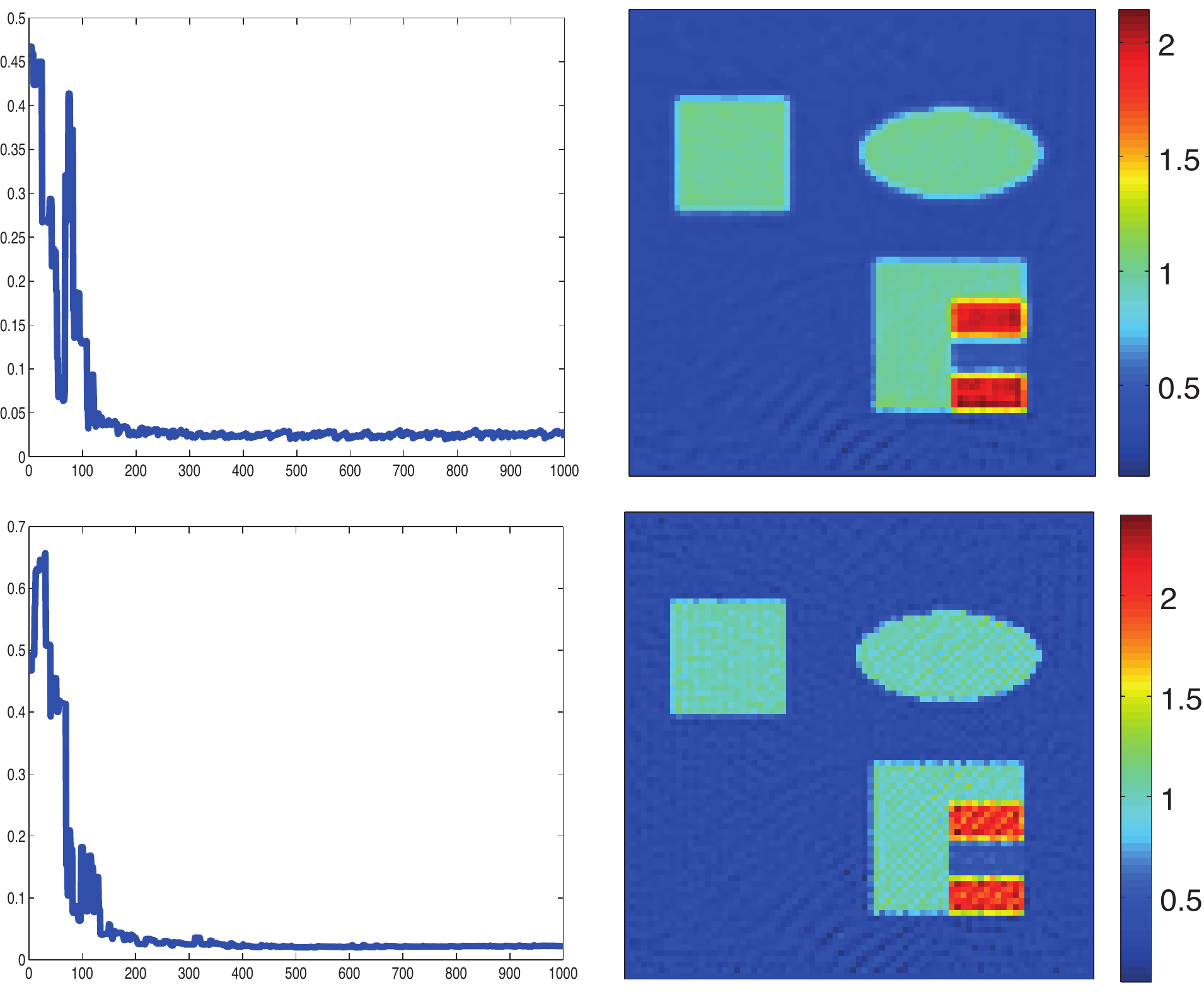}
\caption{{{Reconstruction results based on the novel MULL formulation} \eqref{eq:mull}.}\label{fig:MLIP}
\textbf{Top}: MULL-proximal stochastic gradient method based  on the decomposition  \eqref{eq:pen}.
\textbf{Bottom}: MULL-projected stochastic gradient method based  on the decomposition  \eqref{eq:pen2}.
The left images show the relative reconstruction errors of the reconstructed absorption coefficient as a function of the number of iterations, whereas the right pictures show the results after the last iterations. {(The phantom is as described in Figure~\ref{fig:phantom}.)}}
\end{figure}

\end{itemize}

\begin{remark}\label{rem:effort}
Note that in the stochastic gradient methods for the novel MULL formulation of QPAT
 calculating the matrix vector product $M^{(h)} \cdot \Phi$ is the most costly part.
 In contrast, the standard formulation \eqref{eq:ipi} requires the solution of the system $M^{(h)}c{(h)} = b^{(h)}$. Since the matrix $M^{(h)}$ is sparse only in its spatial domain, this is very
 time-consuming. On the other hand, the  matrix $M^{(h)}$ (which is  a discretization of $ \theta \cdot \nabla_x +  \mua  + \mus (\Io-\Ko)$) has a   simple dependence on  the variables
 $\mu_a, \mu_s$. We therefore can compute the velocity entries of $M^{(h)}$ at the beginning of the iterative process  to save computation time.
 \end{remark}

The reconstruction times for the final reconstructions using all methods described above
are shown in Table~\ref{tab:reconstructiontimes}. For the standard formulation of QPAT, the reconstruction times
seem to be in accordance with reported results using gradient or Newton-type methods for QPAT
(see, for example, \cite{SarTarCoxArr13}.)  The methods based on the new
MULL formulation \eqref{eq:mull}  (after 1000 iterations)
are faster than the methods based on
the standard formulation  \eqref{eq:ipi} of QPAT (after 10, respectively, 40 iterations).
From~the relative reconstruction errors shown in Figures \ref{fig:IP} and \ref{fig:MLIP}, one notices that,
opposed to the methods based on \eqref{eq:ipi},
 the~methods using the MULL formulation could even be stopped much earlier while still obtaining a comparable
 reconstruction quality. We roughly estimate a speedup of a factor 10 using the novel MULL formulation instead of the
standard formulation of QPAT.

\begin{table}[htb]
\centering
\begin{tabular}{ c  @{\hspace{0.25cm}} c @{\hspace{0.25cm}} c  @{\hspace{0.25cm}} c   @{\hspace{0.25cm}}  c   }
\toprule
\textbf{Algorithm}   &  \textbf{Model} & \textbf{Update}  &  \textbf{ No. Iterations}  &   \textbf{Reconstruction  Time}  \\
\midrule
Proximal gradient  & \eqref{eq:ipi} &  \eqref{eq:PGla}   & 10   & \SI{27.2}{h}    \\
Proximal stochastic  gradient  & \eqref{eq:ipi} &  \eqref{eq:proxsg}  &  30  &  \SI{24.4}{h}   \\
MULL-proximal stochastic gradient & \eqref{eq:mull}  &  \eqref{eq:mullproj} & 1000  &    \SI{14.7}{h} \\
MULL-projected stochastic gradient   & \eqref{eq:mull}  & \eqref{eq:mullprox} & 1000 &   \SI{11.9}{h} \\
\bottomrule
\end{tabular}
\caption{Reconstruction times for all methods. Recall that one iteration  of the proximal stochastic gradient method is approximately four times cheaper than
one iteration of the full proximal gradient method (both based on \eqref{eq:ipi}). Further recall
that one step in the methods based on the {MULL} formulation \eqref{eq:mull} is much less time consuming
 than for the methods  based on  \eqref{eq:ipi}; see Remark~\ref{rem:effort}.
 \label{tab:reconstructiontimes}}\end{table}

In Figure \ref{fig:noisy}, we show results for  noisy data using the proximal gradient method based on
the standard formulation~\eqref{eq:ipi} (top) and the proximal stochastic gradient method using the MULL formulation for QPAT (bottom).
The regularization parameter is  chosen as in the exact data case. {Finally, in Figure {\ref{fig:two}}, we show reconstruction results using only
two consecutive illuminations applied from the top and from the left with noisy data. We use 10 iterations for
the proximal gradient algorithm based on {\eqref{eq:ipi}} (shown in left image in Figure {\ref{fig:two}}) and 500 iterations for the stochastic gradient algorithms based on the  MULL formulation~{\eqref{eq:mull}} (shown in the right image in Figure {\ref{fig:two}}).}

\begin{figure}[htb]\centering
\includegraphics[width=0.85\textwidth]{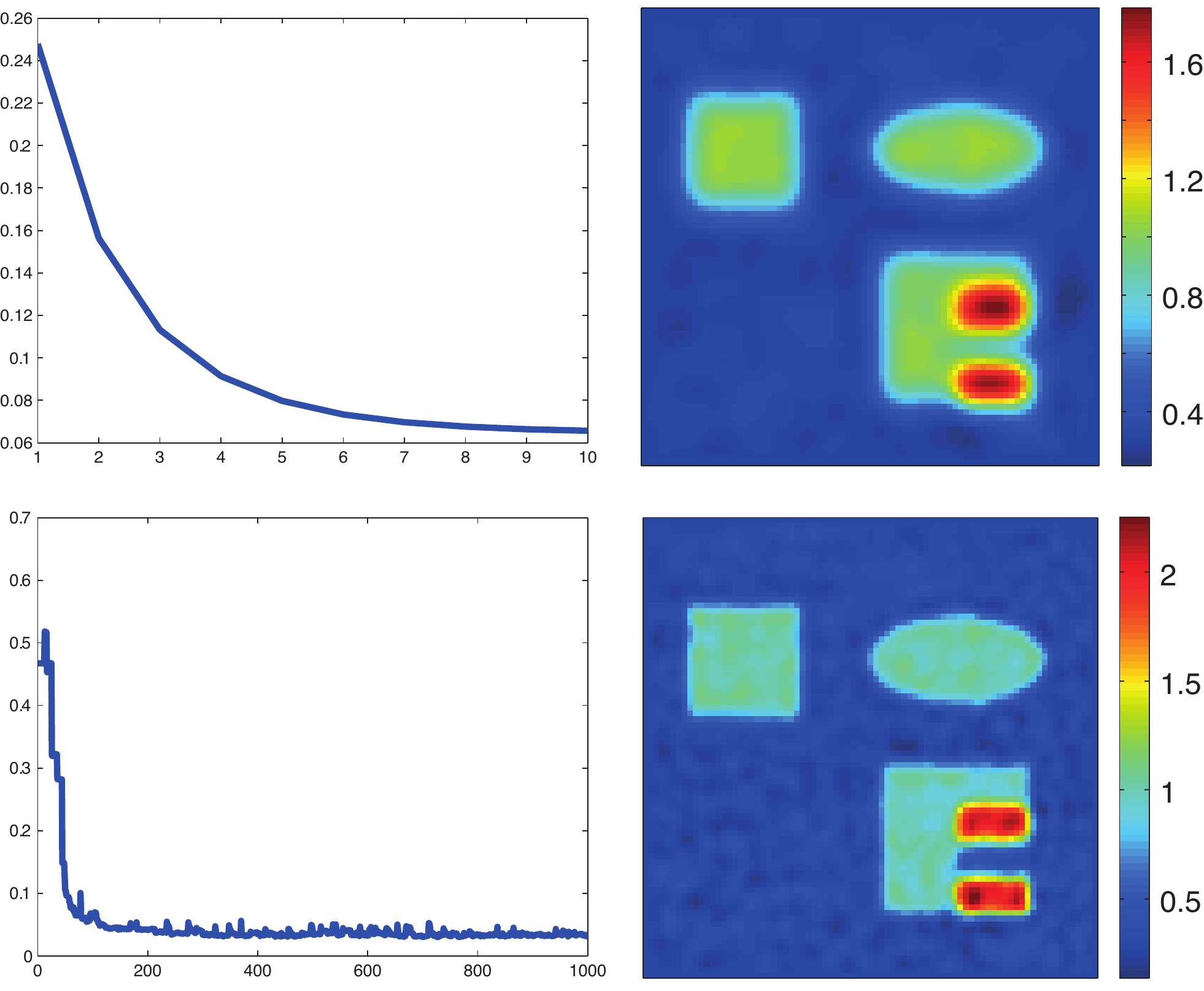}
\caption{{{Reconstruction results from noisy data.}}\label{fig:noisy}
\textbf{Top}: Proximal gradient method based on  \eqref{eq:ipi}.
\textbf{Bottom}: MULL-proximal stochastic gradient method.
The left images show the relative reconstruction errors of the reconstructed absorption coefficient as
a function of the number of iterations, whereas the right pictures show the results after the last iterations. {(The phantom is as described in  Figure~\ref{fig:phantom}.)}}
\end{figure}

\begin{figure}[htb]\centering
\includegraphics[width=\textwidth]{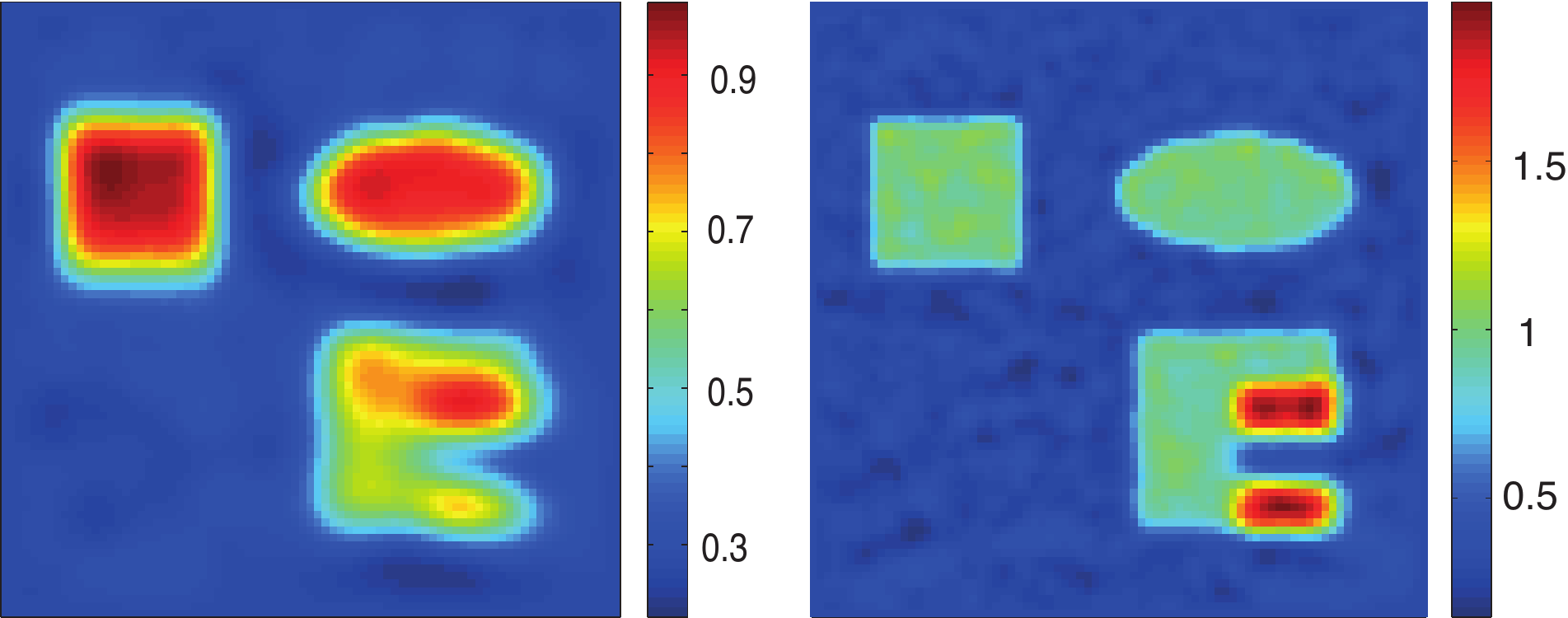}
\caption{{{Reconstruction results from noisy data for two illuminations}.}\label{fig:two}
\textbf{Left}: proximal gradient method based on  \eqref{eq:ipi} using 10 iterations.
\textbf{Right}: MULL-proximal stochastic gradient method using only 500 iterations.
The phantom is as described in Figure~\ref{fig:phantom} and, for the
reconstruction methods, we use two consecutive illuminations (from the top and from the left). The reconstruction time has been about \SI{14}{h} for the method based on the
standard formulation  \eqref{eq:ipi} and \SI{3}{h} for the proposed MULL-proximal stochastic gradient method.}
\end{figure}

\section{Conclusions}
\label{sec:conclusion}

In this paper, we developed  efficient proximal stochastic gradient methods for image reconstruction in multi-source
QPAT. We used the RTE as an accurate model for light transport and employed the single stage approach
for QPAT introduced in \cite{HalNeuRab15}. One class of the proximal stochastic gradient methods has
been developed based on the standard formulation for QPAT given in~\eqref{eq:ipi}. Additionally, we
developed another class using proximal stochastic gradient methods for the new MULL formulations
\eqref{eq:mull} and \eqref{eq:consTik} for QPAT.
Besides proposing proximal stochastic gradient methods for QPAT, we also
consider the formulations  \eqref{eq:mull} and \eqref{eq:consTik} as the main contributions of the
present article.
{These new formulations avoid the time-consuming evaluation of the RTE at each iteration and allow for treating the QPAT problem as a constrained optimization problem, which enables the use of a variety of numerical algorithms. Here, we used a penalty
approach in  combination with stochastic gradient methods for the solution.}
Future work will be done in the direction of developing new algorithms based on \eqref{eq:mull} and
\eqref{eq:consTik}. Additionally, we will investigate the use of different regularization
terms in~\eqref{eq:consTik}. Finally,~the  theoretical convergence analysis of proximal gradient algorithms and other iterative algorithms for solving \eqref{eq:mull} will be the subject of future research.

\section*{Acknowledgments}

M.H. acknowledges support of the Austrian Science Fund (FWF),
project P 30747. 
S.R. is a recipient of a {DOC} Fellowship of the
Austrian Academy of Sciences (OEAW)  and acknowledges support of the OEAW for his PhD project.
We thank  the referees  for helpful comments that increased the readability of the
this manuscript, as well as pointing out some interesting references  to us.

\end{document}